\documentclass{amsart}

\usepackage{amsrefs} 

\usepackage{amsmath}
\usepackage{amssymb}
\usepackage{mathrsfs}  
\usepackage{hyperref}
\usepackage{enumitem}

\newtheorem{theorem}{Theorem}[section]
\newtheorem{lemma}[theorem]{Lemma}
\newtheorem{corollary}[theorem]{Corollary}
\newtheorem{proposition}[theorem]{Proposition}

\theoremstyle{definition}
\newtheorem{definition}[theorem]{Definition}

\newtheorem{convent}[theorem]{Convention}

\newcommand\R{\mathbb{R}}
\newcommand\Z{\mathbb{Z}}

\newcommand\C{\mathbb{C}}
\newcommand\N{\mathbb{N}}
\newcommand\hs{\mathfrak I_2}
\newcommand\I{\mathfrak I}

\newcommand{\qtq}[1]{\quad\text{#1}\quad}

\newcommand\eps{\varepsilon}
\newcommand{\vk}{\varkappa}

\newcommand{\sbrack}[1]{^{[#1]}}

\newcommand{\ham}{H} 
\newcommand{\efg}{\text{\sl EF}}

\let\Re=\undefined\DeclareMathOperator{\Re}{Re}
\let\Im=\undefined\DeclareMathOperator{\Im}{Im}
\DeclareMathOperator{\Id}{Id}

\DeclareMathOperator{\tr}{tr}
\DeclareMathOperator{\sgn}{sgn}
\newcommand{\hlm}{\textsf M}

\newcommand\op{\mathrm{op}}
\newcommand{\ess}{\mathrm{ess}}

\newcommand{\mc}[1]{\mathcal{#1}}
\newcommand{\ms}[1]{\mathscr{#1}}
\newcommand{\ol}[1]{\overline{#1}}
\newcommand{\wh}[1]{\widehat{#1}}
\newcommand{\wt}[1]{\widetilde{#1}}
\newcommand{\norm}[1]{\left\lVert#1\right\rVert}
\newcommand{\snorm}[1]{\lVert#1\rVert}
\newcommand{\bnorm}[1]{\big\lVert#1\big\rVert}

\newcommand{\fwh}{\,\widehat{\ }\,}  

\numberwithin{equation}{section}

\allowdisplaybreaks

\begin{document}

\title[Scaling-critical well-posedness for CCM]{Scaling-critical well-posedness for\\continuum Calogero--Moser models}

\author[R.~Killip]{Rowan Killip}
\address{Department of Mathematics, University of California, Los Angeles, CA 90095, USA}
\email{killip@math.ucla.edu}

\author[T.~Laurens]{Thierry Laurens}
\address{Department of Mathematics, University of Wisconsin--Madison, WI, 53706, USA}
\email{laurens@math.wisc.edu}

\author[M.~Vi\c{s}an]{Monica Vi\c{s}an}
\address{Department of Mathematics, University of California, Los Angeles, CA 90095, USA}
\email{visan@math.ucla.edu}

\begin{abstract}
We prove that the focusing and defocusing continuum Calogero--Moser models are well-posed in the scaling-critical space $L^2_+(\R)$.  In the focusing case, this requires solutions to have mass less than that of the soliton.
\end{abstract}

\maketitle


\section{Introduction} \label{s:intro}

The goal of this paper is to address the well-posedness theory of the following two dispersive equations:
\begin{equation}\label{CM-DNLS}\tag{CCM}
i\tfrac{d}{dt} q = -q'' \pm 2iq C_+\big( |q|^2\big)' .
\end{equation}
As we will explain, a plus sign in front of the nonlinearity corresponds to a focusing equation;  a minus sign yields a defocusing model.  Throughout the paper, the $\pm$ and $\mp$ signs will be reserved for this dichotomy: the upper sign will correspond to the focusing case and the lower sign to the defocusing case.

The unknown field $q(t,x)$ appearing in \eqref{CM-DNLS} is a complex-valued function of $x\in\R$.  We further demand that $q(t)$ belongs to the Hardy space
$$
L^2_+(\R) := \{f\in L^2(\R):\, \widehat{f}(\xi) = 0 \text{ for } \xi<0\} .
$$
The operator $C_+$ appearing in the nonlinearity of \eqref{CM-DNLS} denotes the Cauchy--Szeg\H o projection from $L^2(\R)$ onto $L^2_+(\R)$; see \eqref{FT} and \eqref{E:Cauchy-Szego}.

The defocusing \eqref{CM-DNLS} appears first as a special case of the `intermediate nonlinear Schr\"odinger equation' introduced in \cite{Pelinovsky1995}.   Concretely, this equation is derived as a model for the envelope of an approximately monochromatic wave packet at the interface between otherwise quiescent fluid layers of infinite total depth, that is, it provides a modulation theory for the setting of the Benjamin--Ono equation.  In addition to deriving a model for the intermediate depth case, Pelinovsky also discovered the important nonlocal structure of the nonlinearity, which had been overlooked in the prior work \cite{Tanaka} on the infinite-depth problem.

Our interest in \eqref{CM-DNLS} was first sparked by the recent work \cite{GerardLenzmann} centered on the focusing case.  This model was formally derived in \cite{Abanov2009} as a continuum limit of the famous Calogero--Moser particle system \cite{Calogero1,Calogero2,Moser}: 
\begin{align}\label{CM}\tag{CM}
\frac{d^2 x_j}{dt^2}  = \sum_{k\neq j} \frac{1}{(x_j-x_k)^3}.
\end{align}
This discrete completely integrable system describes the dynamics of a gas of particles interacting pairwise through an inverse square potential.  It was subsequently discovered that the system remains completely integrable for interaction potentials derived from the Weierstrass $\wp$ function.  Included in this class is the cosecant-squared potential, which arises naturally in the periodic setting.  This latter model was first championed by Sutherland \cite{Sutherland1,Sutherland2} in the quantum mechanical setting.  Correspondingly, the periodic analogue of \eqref{CM-DNLS} has been studied under the name Calogero--Sutherland derivative NLS; see \cite{Badreddine2023}.

Whether derived as a modulation theory for the Benjamin--Ono equation or as a continuum limit of the Calogero--Moser system, there can be no surprise that the \eqref{CM-DNLS} equations are completely integrable.   This fact will ultimately play a key role in what follows.

The most physically intuitive conserved quantities are 
\begin{align*}
M(q) := \int |q|^2\,dx \qtq{and} P(q) := \int -i\ol{q}q' \mp \tfrac{1}{2}|q|^4 \,dx, 
\end{align*}
which represent mass and momentum, together with the Hamiltonians
\begin{align}\label{E:Hamil}
H_\pm(q)&:=\frac12\int_{\R} \bigl| q'\mp iqC_+(|q|^2) \bigr|^2 \, dx .
\end{align}
We caution the reader that these identifications rely on the non-standard Poisson structure naturally associated to \eqref{CM-DNLS}; see \cite{GerardLenzmann}.  This exotic structure originates from the use of a gauge transformation to simplify the underlying phase space.

The \eqref{CM-DNLS} equations are  \emph{mass-critical}: both the class of solutions and the conserved mass $M(q)$ are invariant under the scaling
$$
q(t,x) \mapsto q_\lambda(t,x) := \sqrt{\lambda} q\bigl(\lambda^2t, \lambda x\bigr) \quad\text{for $\lambda>0$}.
$$
These considerations make $L^2_+(\R)$ the most natural space in which to address the well-posedness problem.  In this sense, we are able to give a complete solution to the defocusing problem:

\begin{theorem}[Defocusing case]\label{t:wpd}
The defocusing \eqref{CM-DNLS} equation is globally well-posed in the space $L^2_+(\R)$.  Moreover, if $Q\subset L^2_+(\R)$ is a bounded and equicontinuous set, then the set of orbits
\begin{equation}\label{Q*d}
Q^* := \{ e^{tJ\nabla\ham_-} q : q\in Q,\ t\in\R \} 
\end{equation}
is bounded and equicontinuous in $L^2_+(\R)$.
\end{theorem}

The analogous theorem for periodic initial data was proved recently by Badreddine \cite{Badreddine2023}, whose influence on our work will be discussed in due course.  One point in which we must diverge is in the treatment of equicontinuity. 

\begin{definition}[Equicontinuity]\label{d:equicontinuity}
Fix $s\in \R$. A bounded set \(Q\subset H^s_+(\R)\) is said to be \emph{equicontinuous} in the $H^s(\R)$ topology if
\[
\limsup_{\delta\to0} \ \sup_{q\in Q} \ \sup_{|y|<\delta}\|q(\cdot +y) - q(\cdot)\|_{H^s(\R)} = 0.
\]
\end{definition}

Replacing the $H^s(\R)$ norm here by the supremum norm would yield the familiar notion appearing in the Arzel\`a--Ascoli Theorem.  Indeed, boundedness, tightness, and equicontinuity constitute necessary and sufficient conditions for a set to be $H^s(\R)$-precompact.  Equicontinuity may also be understood via Plancherel's Theorem: a bounded set $Q\subset H^s(\R)$ is equicontinuous if and only if
\begin{equation}\label{equicty 1}
\lim_{ \kappa\to\infty}\ \sup_{q\in Q}\int_{|\xi|\geq \kappa} |\wh{q}(\xi)|^2 (|\xi|+1)^{2s}\, d\xi =0,
\end{equation}
which is to say, the family of Fourier transforms is tight.

Equicontinuity plays a central role in the treatment of scaling-critical problems.  It is precisely this property that prevents physical-space concentration and so blowup.  Due to its scaling-critical nature, the conservation of mass is completely insensitive to any changes of scale.
On the other hand, control of higher regularity explicitly ensures equicontinuity.  Such a higher regularity theory has been developed in \cite{GerardLenzmann} in a manner applicable to both the focusing and defocusing models; see, for example, Theorem~\ref{T:GL} below.  To properly describe this, we must first discuss some startling phenomenology discovered in \cite{GerardLenzmann} regarding the focusing model.

The focusing \eqref{CM-DNLS} equation admits soliton solutions: 
\begin{equation}\label{CCMsoliton} 
q(t,x) = \sqrt{\lambda} R ( \lambda x+x_0) \qtq{with} \lambda>0, \quad x_0\in \R, \qtq{and} R(x) := \tfrac{\sqrt{2}}{x+i}.
\end{equation}
These (stationary) solitons all have the same mass, namely, $M(R)=2\pi$.  Moreover, the focusing \eqref{CM-DNLS} admits multisoliton solutions; these have been analyzed in detail in \cite{GerardLenzmann}.

The multisoliton solutions are rational functions of $x$ at every moment of time.  Their poles and residues both evolve in time, with the former satisfying a complex Calogero--Moser system.  Such a solution is called an $N$-soliton if there are $N$ poles.   Their mass is quantized; it is $2\pi N$, that is, $N$ times the mass of the one-soliton $R$.  The most striking discovery, however, is that for $N\geq2$, these $N$-solitons exhibit frequency cascades as $t\to\pm\infty$.  As a consequence, for every $s>0$, the $H^s(\R)$ of these smooth solutions is unbounded as $t\to\pm\infty$.

By contrast, \cite{GerardLenzmann} shows that smooth solutions with mass below that of the one-soliton do not exhibit such norm growth: 

\begin{theorem}[\cite{GerardLenzmann}]\label{T:GL}
Let $n\geq 1$ be an integer.  The focusing \eqref{CM-DNLS} is locally well-posed in $H^n_+(\R)$.  Furthermore, it is globally well-posed on $\{ q\in H^n_+(\R) : M(q)<2\pi\}$ and these solutions satisfy
\begin{equation}\label{GLbd}
\sup_{t\in \R} \|q(t)\|_{H^n_+(\R)}<\infty.
\end{equation}
The defocusing equation is globally well-posed in $H^n_+(\R)$, without any mass constraint, and all solutions also satisfy \eqref{GLbd}.
\end{theorem}

Strictly speaking, the paper \cite{GerardLenzmann} does not discuss well-posedness in the defocusing case; nevertheless, applying the arguments described there yields the results stated above.  The paper also demonstrates local well-posedness for regularities $s>\tfrac12$ and shows that smooth initial data with mass exactly equal to $2\pi$ leads to a global smooth solution.

The specific number $2\pi$ in Theorem~\ref{T:GL} originates from the sharp constant in the inequality \eqref{op est 4}.  This threshold coincides with the mass of the soliton, precisely because $R$ is an optimizer for the inequality \eqref{op est 4}.

When considering initial data from $H^\infty_+(\R)$, Theorem~\ref{T:GL} allows us to discuss \emph{the} solution to \eqref{CM-DNLS}.  The first question we ask about such solutions is one of equicontinuity.  The explicit form of the two-soliton solutions presented in \cite{GerardLenzmann} shows that they do not have $L^2$-equicontinuous orbits.  Thus, the global-in-time equicontinuity property of Theorem~\ref{t:wpd} does not extend to the focusing case.  On the other hand, the results presented in Theorem~\ref{T:GL} strongly suggest that global-in-time equicontinuity may hold under a mass constraint, even though the arguments used in \cite{GerardLenzmann} do not yield a proof of this.  This is our first result on the focusing problem:

\begin{theorem}[Equicontinuity for the focusing \eqref{CM-DNLS}]\label{t:equicty}
If $Q\subset H^\infty_+(\R)$ satisfies 
\begin{equation}\label{M<M*}
\sup_{q\in Q} \norm{q}_{L^2}^2 < 2\pi 
\end{equation}
and is equicontinuous in $L^2(\R)$, then the set of orbits
\begin{equation}\label{Q*f}
Q^* := \{ e^{tJ\nabla\ham_+} q : q\in Q,\ t\in\R \} 
\end{equation}
is bounded and equicontinuous in $L^2(\R)$.
\end{theorem}

Recall that the two-soliton solutions demonstrate that this conclusion would fail if the number $2\pi$ in \eqref{M<M*} were replaced by any number larger than $4\pi$.  At this moment, we are not ready to a make a conjecture about the true mass threshold for global-in-time equicontinuity.

As noted above, control on the frequency distribution in a scaling-critical space is an essential ingredient for proving global well-posedness because it precludes space concentration.  For this purpose, it is not necessary that orbits remain equicontinuous globally in time.  It would suffice to know that equicontinuous sets of smooth initial data lead to families of orbits that are equicontinuous on compact time intervals.  This line of reasoning leads to our introduction of the following mass threshold:

\begin{definition}\label{t:M* def}
Let $M^* \in [0,\infty]$ denote the maximal constant so that for any $L^2$ bounded and equicontinuous set $Q\subset H^\infty_+(\R)$ satisfying
\begin{equation*}
\sup_{q\in Q} M(q) < M^* ,
\end{equation*}
the set of partial orbits
\begin{equation}
Q^*_T := \{ e^{tJ\nabla H_\pm} q : q\in Q,\ |t|< T \} 
\label{Q*T}
\end{equation}
is $L^2$-equicontinuous for each choice of $T>0$.
\end{definition}

We are not implicitly assuming here that all choices of smooth initial data lead to a global solution.  In the focusing case, this is an open problem.  Rather, we will show that finite-time blowup must be accompanied by a loss of equicontinuity; see Lemma~\ref{L:horizon} for details.

In the defocusing case, Theorem~\ref{t:wpd} shows that $M^*=\infty$.  Indeed, we will first prove that $M^*=\infty$ in Section~\ref{S:3} and then later use this as an ingredient in proving well-posedness in Section~\ref{S:5}.

In the focusing case, Theorem~\ref{t:equicty} guarantees $M^*\geq 2\pi$.  The analysis of multisoliton solutions in \cite{GerardLenzmann} places no restrictions on $M^*$; indeed, mass moves to high frequencies in a manner that is linear in time.

The main rationale for introducing the equicontinuity threshold $M^*$ is that it sets the proper generality for our well-posedness analysis of the focusing problem.  It will also allow us to treat the  focusing and defocusing problems in a parallel manner.

\begin{theorem}[Global well-posedness in the focusing case]\label{t:wpf}
The focusing \eqref{CM-DNLS} equation is globally well-posed in the space
\begin{equation*}
B_{M^*} = \{ q\in L^2_+(\R) : \norm{q}_{L^2}^2 < M^* \}.
\end{equation*}
\end{theorem}

The analogous result in the periodic setting was obtained in \cite{Badreddine2023}.  The relation between the arguments used here and in \cite{Badreddine2023} will be elaborated below.

In Section~\ref{S:6}, we will show how to deduce well-posedness at higher regularities from our scaling-critical results.  This rests on a new technique for proving $H^s$-equicontinuity that may be of independent interest.

As mentioned earlier, the complete integrability of \eqref{CM-DNLS} will play a central role in our analysis.  This integrability will manifest in two ways: through the Lax pair and an explicit formula (in the sense of G\'erard).

The Lax pair we employ is the following:
\begin{equation}\label{Lax pair}
\mc L = -i\partial \mp q C_+\ol{q} \qtq{and} \mc P = i\partial^2 \pm 2q\partial C_+\ol{q} . 
\end{equation}
A Lax pair formulation of the full intermediate NLS was introduced in \cite{PelinovskyG1995}; recall that one may recover the defocusing \eqref{CM-DNLS} as an infinite depth limit of this model.  The Lax pair \eqref{Lax pair} is a small modification of the one presented in \cite{GerardLenzmann}, which uses
\begin{equation}\label{bad P}
\widetilde{\mc P}:= \mc P + i\mc L_q^2 = \pm qC_+\ol{q}' \mp q' C_+\ol{q} + iqC_+|q|^2C_+\ol{q} .
\end{equation}
Evidently, the change to $\mc P$ has no effect on the commutator relation.  In choosing to center our analysis on \eqref{Lax pair}, we were very much informed by our prior work \cite{Killip2023} on the Benjamin--Ono equation.  Concretely, we favor $\mc P$ because it enjoys two advantageous properties : First, \eqref{CM-DNLS} may be written as
\begin{equation}
\tfrac{d}{dt} q = \mc P q,
\label{one true P}
\end{equation}
and secondly, $\mc P 1 = 0$.  As the constant function $1$ lies outside the natural Hilbert space setting, this latter relation requires some further interpretation; see \eqref{E:P1=0} for a precise statement.

The second manifestation of complete integrability is an explicit formula of a type championed by G\'erard and collaborators \cite{MR3301889,MR4662323,Gerard2023Szego}:

\begin{theorem}\label{t:magic formula}
Let $q(t)$ denote the global solution to~\eqref{CM-DNLS} with initial data $q^0\in L^2_+(\R)$ satisfying $M(q^0)<M^*$. Then 
\begin{equation}\label{magic formula}
q(t,z) = \tfrac{1}{2\pi i} I_+\big\{ (X + 2t\mc{L}_{q^0} - z)^{-1} q^0 \big\}
\end{equation}
for all $z$ with $\Im z>0$.  Here $q(t,z)$ is defined via harmonic extension; see \eqref{PIF11}.
\end{theorem}

For the full definitions of the linear functional $I_+$ and of the operator $X$, see Section~\ref{S:4}.  Naively speaking, $I_+$ represents integration over the whole line, while $X$ generalizes the operator of multiplication by $x$.  Note that $I_+$ is not a bounded linear functional on $L^2_+(\R)$, nor can multiplication by $x$ be interpreted as a self-adjoint operator in Hardy space.  

It is worth recalling that an explicit form of the solution to the Calogero--Moser system \eqref{CM} was given long ago in \cite{MR455039}.  Concretely, it was shown that one may describe particles evolving according to this system as the eigenvalues of a straight line trajectory in the space of symmetric matrices.  This is evidently not a very close analogue of \eqref{magic formula}, which speaks directly to the subtlety of the manner in which \eqref{CM-DNLS} arises as a continuum limit of the particle system.

The analogue of \eqref{magic formula} appropriate to the periodic \eqref{CM-DNLS} was introduced in \cite{Badreddine2023} and played a central role in that analysis.  

In Theorem~\ref{T:smooth EF} we show that the explicit formula \eqref{magic formula} holds for solutions with smooth well-decaying initial data.  This is the form that we will employ in demonstrating scaling-critical well-posedness.  Once this is achieved, we may then extend this formula to all $L^2_+(\R)$ solutions by using the resulting continuity of the data-to-solution map in concert with the continuity of the right-hand side that will be demonstrated using the tools developed as a part of our analysis. 

\subsection*{Overview of the proofs} The central theme of this paper is to demonstrate how to prove well-posedness by synthesizing explicit representations of the type \eqref{magic formula} with the tools and techniques developed as part of the method of commuting flows \cite{MR4304314,harropgriffiths2022sharp,harropgriffiths2023global,Killip2023,MR4628747,MR3990604,MR3820439}.  In particular, the methods we employ to prove equicontinuty draw from these earlier works.  On the other hand, we will employ no regularized flows or commuting Hamiltonians in this paper.

The question central to $L^2_+(\R)$ well-posedness is this:  Does any sequence of smooth and well-decaying initial data $q^0_n$ that converges in $L^2_+(\R)$ lead to a sequence of solutions that converges in $C_t(I; L^2_+(\R))$ on any compact time interval $I$?

By the conservation of $M(q)$, such sequences of solutions will always admit subsequential limits in the weak topology (pointwise in time).  A first step forward is to show that there is a unique such subsequential limit and, correspondingly, one has weak convergence without passing to a subsequence.  This is one of the roles played by the explicit formula in~\cite{Badreddine2023} and in this paper; see Corollary~\ref{C:weak}.

To upgrade weak convergence to strong convergence one must preclude a loss of mass.  In the periodic setting of \cite{Badreddine2023}, this is a question of equicontinuity.  On the line, one must also demonstrate tightness.

Our route to controlling the equicontinuity properties of solutions is through the spectral theory of the Lax operator $\mc L$.  Recall that (formally at least) the Lax equation ensures that $\mc L_{q(t)}$ and $\mc L_{q(0)}$ remain unitarily equivalent.  In the periodic setting, the Lax operator has compact resolvent and so its spectral properties are encoded through the associated sequence of eigenvalues and their eigenvectors.  In \cite{Badreddine2023}, loss of mass is precluded by demonstrating continuity properties of these eigenvalues/vectors.

In the whole line setting, the spectral theory is more complicated; indeed, even the spectral type is unknown at this time.  In Section~\ref{S:2}, we construct the Lax operator as a self-adjoint operator for general $q\in L^2_+(\R)$.  This involves an improvement on the earlier analysis of \cite{GerardLenzmann}:  In Lemma~\ref{L:rc}, we prove that $\mc L_q$ is a relatively compact perturbation of the case $q\equiv 0$, rather than an infinitesimally \emph{form} bounded perturbation as in \cite{GerardLenzmann}. The advantages of this small improvement accumulate as we progress.

In Lemmas~\ref{t:beta} and ~\ref{L:conserv law}, we show that the difference between the resolvents
\begin{align*}
R(\kappa,q(t)) = (\mc L_{q(t)} + \kappa)^{-1} \qtq{and} R_0(\kappa):=(\mc L_0 + \kappa)^{-1} 
\end{align*}
is trace class; moreover, the trace is a conserved quantity.  Consequently, 
\begin{equation}\label{beta defn}
\beta(\kappa,q) := M(q) \mp 2\pi\kappa \tr\{ R(\kappa)-R_0(\kappa) \} 
\end{equation}
is also conserved.  In Section~\ref{S:3}, we use $\beta(\kappa,q)$ to prove our equicontinuity results.  

Lemma~\ref{L:equi via K} describes the connections between equicontinuity and the large-$\kappa$ behaviour of $\beta(\kappa,q)$.  We see that $\beta(\kappa,q)$ converges to zero uniformly on equicontinuous sets. However, to prove equicontinuity, we must show convergence of the \emph{quadratic} part of $\beta(\kappa,q)$.  By virtue of a favorable sign of the higher-order terms in $\beta(\kappa,q)$, this approach quickly yields equicontinuity in the defocusing case.

In the focusing case, the sign is unfavorable!  Nevertheless, by exploiting the inequality \eqref{op est 4} of \cite{GerardLenzmann} in concert with certain operator analysis (which appears to be novel), we are able to subordinate the higher-order terms in $\beta(\kappa,q)$ to the quadratic term and so deduce equicontinuity up to the $2\pi$ threshold.

The climax of Section~\ref{S:4} is the verification of the explicit formula \eqref{magic formula} for smooth and well-decaying initial data; see Theorem~\ref{T:smooth EF}.  Much of this section is devoted to developing the relevant operator theory in a manner that will allow us to consider $L^2_+$ limits of such solutions later.

In Section~\ref{S:5}, we complete the proof of well-posedness in $L^2_+(\R)$.  Recall that our goal is to show that $L^2_+(\R)$-Cauchy sequences of smooth well-decaying initial data produce solutions that converge in $C_tL^2_x$ on compact time intervals.  At this moment in the argument, we know weak convergence at each moment of time.  We upgrade this by demonstrating compactness.  This requires three inputs: equicontinuity and tightness of orbits as functions of $x$ and thirdly, $L^2_+(\R)$-equicontinuity in time.

Spatial equicontinuity is settled already in Section~\ref{S:3}.  As we will see, this equicontinuity is then used to aid in verifying the other two requirements.  See Lemma~\ref{L:equi in t} for the treatment of equicontinuity in time and Proposition~\ref{t:tight 2} for tightness.  To prove the latter, we rely again on the explicit formula.

Section~\ref{S:5} concludes with the proof of Theorem~\ref{t:magic formula}.  This combines the well-posedness proved earlier in the section with the operator theory developed in Section~\ref{S:4}.  

Finally, in Section~\ref{S:6}, we extend well-posedness to $H^s_+(\R)$ for $0\leq s<1$.  The cornerstone of this extension is the demonstration of equicontinuity of orbits in $H^s_+(\R)$.  Breaking from previous works, we introduce a method based on Loewner's Theorem on operator monotone functions (see the monograph \cite{MR3969971}).  We believe this represents an elegant and efficient approach to this question that will prove useful in treating other completely integrable systems.

\subsection*{Acknowledgements} R.K. was supported by NSF grant DMS-2154022;  M.V. was supported by NSF grant DMS-2054194.  The work of T.L. was also supported by these grants.

\subsection{Notation}
Throughout this paper, we employ the standard notation $A \lesssim B$ to indicate that $A \leq
CB$ for some constant $C > 0$; if $A \lesssim B$ and $B \lesssim A$, we write $A \approx B$.
Occasionally, we adjoin subscripts to this notation to indicate dependence of the constant
$C$ on other parameters; for instance, we write $A \lesssim_{\alpha, \beta} B$ when $A \leq
CB$ for some constant $C > 0$ depending on $\alpha, \beta$.

Hardy spaces will be denoted $L^p_+= L^p_+(\R)$.  These are the (closed) subspaces of $L^p(\R)$ comprised of those functions $f$ whose Poisson integral
\begin{align}\label{PIF11}
    f (z) :=  \int \frac{\Im z}{\pi |x-z|^2} f(x) \,dx \qtq{defined for} \Im z >0 
\end{align}
is holomorphic (in the upper half-plane).  By H\"older's inequality,
\begin{align}\label{266}
    \bigl| f (z) \bigr| \lesssim  (\Im z)^{-\frac1p}\| f \|_{L^p(\R)}   
\end{align}
for any $1\leq p \leq \infty$.

For \(s\in \R\) we define the Sobolev spaces $H^s(\R)$ as the completion of $\mathcal S(\R)$ with respect to the norm
\begin{equation*}
\norm{f}_{H^s(\R)}^2 = \int  (|\xi|+1)^{2s} |\wh{f}(\xi)|^2\,d\xi.
\end{equation*}
The Hardy--Sobolev spaces $H^s_+(\R)$ comprise those functions in $H^s(\R)$ whose Fourier transform is supported on $[0,\infty)$.

One important property of the Hardy--Sobolev spaces that is not enjoyed by the pure Sobolev spaces is that products are well defined even at negative regularity.   For example,
\begin{equation}\label{277}
\norm{fg}_{H^{-5}_+(\R)} \lesssim \norm{f}_{H^{-2}_+(\R)} \norm{g}_{H^{-2}_+(\R)} .
\end{equation}
The proof is elementary; see \cite[Lem.~2.2]{Killip2023} for details.

Our convention for the Fourier transform is
\begin{align}\label{FT}
\widehat f(\xi) = \tfrac{1}{\sqrt{2\pi}} \int_\R e^{-i\xi x} f(x)\,dx  \qtq{so that } f(x) = \tfrac{1}{\sqrt{2\pi}} \int_\R e^{i\xi x} \widehat f(\xi)\,d\xi.
\end{align}
With this definition, 
\begin{align}\label{268}
    \|f\|_{L^2(\mathbb R)}=\|\widehat f\|_{L^2(\mathbb R)} \qtq{and} \wh{fg}(\xi) = \tfrac{1}{\sqrt{2\pi}}\int \wh f(\xi-\eta)\wh g(\eta)\,d\eta. 
\end{align}

We use the following notation for the Cauchy--Szeg\H{o} projections
\begin{equation}\label{E:Cauchy-Szego}
\wh{C_\pm f}(\xi) := 1_{[0,\infty)}(\pm\xi) \wh{f}(\xi) .
\end{equation}

We will employ Littlewood--Paley decompositions with frequency parameters $N\in 2^\Z$. For a smooth, non-negative function \(\varphi\) supported on $|\xi|\leq 2$ with $\varphi(\xi)=1$ for $|\xi|\leq 1$, we define $P_{\leq N}$ as the Fourier multiplier with symbol $\varphi(\xi/N)$ and then $P_N = P_{\leq N} - P_{\leq N/2}$.  Observe that 
\[
1 = \sum_{N\in 2^\Z} P_N.
\]
We will often adopt the more compact notations $f_N=P_N f$ , $f_{\leq N}=P_{\leq N} f$, and $f_{>N} = [1 - P_{\leq N}]f$.

Throughout the paper, we write $\| \cdot\|_\op$ to denote the norm of an operator acting on the Hilbert space $L^2_+(\R)$.  Similarly, we reserve the notations $\I_1$ and $\I_2$ to denote the trace and Hilbert--Schmidt classes over this same Hilbert space.  For further information on such trace ideals, we recommend the book \cite{TraceIdeals}.

\section{The Lax operator}\label{S:2}

We write $\mc{L}_0 = |\partial| = -i\partial$.  For $\kappa>0$, the resolvent associated with $\mc{L}_0$ is
\begin{equation*}
R_0(\kappa) := (\mc{L}_0+\kappa)^{-1}.
\end{equation*}

In the focusing case, the $2\pi$ mass threshold marks an important transition in the spectral theory of the Lax operator $\mc L_q$ presented in \eqref{Lax pair}.  This particular number originates from the sharpness of the constant in the inequality
\begin{equation}\label{op est 4}
\norm{ C_+(\ol{q}f) }_{L^2} \leq \tfrac{1}{\sqrt{2\pi}} \norm{q}_{L^2} \bnorm{ |\partial|^{1/2} f }_{L^2} ,
\end{equation}
demonstrated already in \cite{GerardLenzmann}.

The inequality \eqref{op est 4} was used in \cite{GerardLenzmann} to construct the operator $\mc L_q$ as an infinitesimally form-bounded perturbation of $\mc L_0$.  This also allowed them to identify the form domain of $\mc L_q$.  Here we will prove the stronger result that $\mc L_q$ is a relatively compact perturbation of $\mc L_0$; this allows us to also identify the domain and the essential spectrum of $\mc L_q$.  The key input is the following:

\begin{lemma}\label{L:rc}  For any $\kappa>0$ and $g,q\in L^2_+(\R)$, the operator
\begin{equation}\label{rc}
 f \mapsto g C_+ \bigl[\ol q R_0(\kappa) f\bigr] \qtq{is compact and} \bigl\| g C_+ \ol q R_0(\kappa) \bigr\|_\op \lesssim \|g\|_{L^2} \|q\|_{L^2}.
\end{equation}
Moreover, if $Q\subseteq L^2_+(\R)$ is bounded and equicontinuous, then
\begin{align}\label{rc equi}
\sup_{q\in Q}\, \bnorm{ C_+\ol{q}\sqrt{R_0(\kappa)} }_{\op}^2 \leq \sup_{q\in Q} \, \bigl\| q C_+ \ol q R_0(\kappa) \bigr\|_\op \to 0 \qtq{as} \kappa\to\infty.
\end{align}
\end{lemma}

\begin{proof}
Given $\xi>0$, we may use \eqref{268} to see that
\begin{align*}
\bigl| \widehat{ \bar q R_0 f }(\xi)  \bigr| \leq \tfrac{1}{\sqrt{2\pi}} \int_\xi^\infty \bigl| \wh q(\eta-\xi) \bigr| \bigl| \wh f (\eta)\bigr| \,\frac{d\eta}{\eta}, 
\end{align*}
for any $\kappa>0$, and consequently,
\begin{align}\label{1477}
\int_h^\infty \bigl| \widehat{ \bar q R_0 f }(\xi)  \bigr| \,d\xi
	\lesssim  \int_h^\infty \tfrac{1}{\eta}\int_0^\eta \bigl| \wh q (\zeta)\bigr|\,d\zeta \bigl| \wh f (\eta)\bigr| \,d\eta
		\lesssim \| \hlm \wh q\|_{L^2([h,\infty))} \| f\|_{L^2(\R)},
\end{align}
where $\hlm$ denotes the Hardy--Littlewood maximal function.

Choosing $h=0$ in \eqref{1477}, we find that
\begin{align*}
\int_E  \bigl| [g C_+ (\ol q R_0 f )](x)  \bigr|^2 dx
	\leq  \| g \|_{L^2(E)}^2 \| C_+ (\ol q R_0 f )\|_{L^\infty(\R)}^2
	\lesssim \| g \|_{L^2(E)}^2  \| q \|_{L^2(\R)}^2 \| f \|_{L^2(\R)}^2,
\end{align*}
uniformly for $E\subseteq\R$.  Taking $E=\R$, this shows that the operator admits the norm bound stated in \eqref{rc}.  By choosing $E$ of the form $\{x: |x| >r\}$, we also see that the image of the unit ball in $L^2_+(\R)$ under the operator $g C_+ \ol q R_0(\kappa)$ is tight in $L^2(\R)$.

As a second application of \eqref{1477} together with considerations of the support of a convolution, we get
\begin{align*}
\int_{2h}^\infty \bigl| [g C_+ (\ol q R_0 f)]&\widehat{\ }(\xi)  \bigr|^2 \,d\xi \\
&\leq  \| \wh g \|_{L^2([h,\infty))}^2 \| \widehat{ \ol q R_0 f } \|_{L^1([0,\infty))}^2
	+ \| \wh g \|_{L^2(\R)}^2 \| \widehat{ \ol q R_0 f } \|_{L^1([h,\infty))}^2 \\
&\lesssim \Bigl[ \| \wh g \|_{L^2([h,\infty))}^2 \| \wh q \|_{L^2(\R)}^2 + \| \wh g \|_{L^2(\R)}^2 \| \hlm \wh q \|_{L^2([h,\infty))}^2  \Bigr]
	\| f \|_{L^2(\R)}^2 
\end{align*}
and so the image of the unit ball in $L^2_+(\R)$ under $g C_+ \ol q R_0(\kappa)$ is also tight on the Fourier side. 

In conclusion, the image of the unit ball  in $L^2_+(\R)$ under the operator $g C_+ \ol q R_0(\kappa)$ is precompact in $L^2(\R)$ and so this is a compact operator.

We now turn our attention to property \eqref{rc equi} of the operator $q C_+ \ol q R_0(\kappa)$.  Each $q\in Q$ may be decomposed in frequency as $q=q_{\leq N}+q_{> N}$.  Accordingly,
\begin{equation*}
q C_+ \ol q R_0(\kappa) = q_{> N} C_+ \ol q R_0(\kappa) + q_{\leq N} C_+ \ol q_{>N} R_0(\kappa) + q_{\leq N} C_+ \ol q_{\leq N} R_0(\kappa).
\end{equation*}
The first two summands are treated using \eqref{rc}:
\begin{equation*}
\bigl\| q_{> N} C_+ \ol q R_0(\kappa) + q_{\leq N} C_+ \ol q_{>N} R_0(\kappa) \bigr\|_\op \lesssim \|q\|_{L^2} \|q_{>N}\|_{L^2}.
\end{equation*}
As $Q$ is equicontinuous, choosing $N$ large makes this small uniformly for all $q\in Q$.

Using Bernstein's inequality, we may bound the last summand as follows:
\begin{equation*}
\bigl\| q_{\leq N} C_+ \ol q_{\leq N} R_0(\kappa) \bigr\|_\op \lesssim \kappa^{-1} \|q_{\leq N}\|_{L^\infty}^2
	\lesssim \kappa^{-1} N \|q\|_{L^2}^2.
\end{equation*}
Irrespective of how $N$ is chosen, this can be made small by choosing $\kappa$ large.

By employing $T^*T$ arguments and complex interpolation, we find
\begin{align*}
\bnorm{ C_+\ol q \sqrt{R_0} }_\op^4 = \bnorm{ \sqrt{R_0} q C_+\ol q \sqrt{R_0} }_\op^2
	\leq \bnorm{ R_0 q C_+\ol q  }_\op \bnorm{ q C_+\ol q R_0 }_\op = \bnorm{ q C_+\ol q R_0 }_\op ^2 .
\end{align*}
In particular, this norm also converges to zero as $\kappa\to\infty$, uniformly for $q\in Q$.
\end{proof}

\begin{proposition}[Lax operator]\label{P:Lax}
For $q\in L^2_+(\R)$, the operator
\begin{equation*}
\mc L_q f = - i f' \mp q C_+(\ol q f) 
\end{equation*}
with domain $D(\mc L_q)=H^1_+(\R)$ is self-adjoint and $\sigma_\ess(\mc L_q) = [0,\infty)$.  Moreover, the mapping $q\mapsto \mc L_q$ is continuous in the norm resolvent topology. 

If $Q\subset L^2_+(\R)$ is bounded and equicontinuous, then there exists $\kappa_0=\kappa_0(Q)>0$ so that
\begin{equation}\label{form comp} 
 \tfrac12  (\mc L_0 + \kappa)   \leq  (\mc L_q + \kappa)  \leq \tfrac32 (\mc L_0 + \kappa), \\
\end{equation}
as quadratic forms, whenever $\kappa\geq \kappa_0;$ moreover,
\begin{equation}\label{equiv norms} 
\bigl\| (\mc L_q + \kappa)^{s} f \bigr\|_{L^2} \approx \bigl\| (\mc L_0+ \kappa)^{s} f \bigr\|_{L^2}
\end{equation}
uniformly for $q\in Q$, $\kappa\geq \kappa_0$, and $-1\leq s \leq 1$.
\end{proposition}

\begin{proof}
Compactness of the operator \eqref{rc} guarantees that $q C_+ \bar q$ is infinitesimally $\mc L_0$-bounded.  Thus, self-adjointness and semi-boundedness follow from the Kato--Rellich Theorem \cite[\S X.2]{Reed1975}.  Through Weyl's Theorem, \cite[\S XIII.4]{Reed1978}, compactness of the operator also shows that the essential spectrum agrees with that of $\mc L_0$.

By \eqref{rc equi}, we may choose $\kappa_0=\kappa_0(Q)>0$ so that $\kappa\geq\kappa_0$ implies
\begin{align*}
\bigl\langle f, q C_+( \ol q f)\bigr\rangle  = \bigl\langle C_+(\ol qf), C_+ (\ol q f)\bigr\rangle
	&\leq \bigl\| C_+ \ol q \sqrt{R_0} \bigr\|_{L^2}^2  \bigl\| \sqrt{\mc L_0 + \kappa}\, f \bigr\|_{L^2}^2\\
		&\leq \tfrac12 \bigl\langle f, (\mc L_0 + \kappa) f  \bigr\rangle,
\end{align*}
and also
\begin{equation*}
\bigl\| q C_+(\ol q f)  \bigr\|_{L^2} \leq \bigl\| q C_+ \ol q R_0(\kappa) \bigr\|_\op \bigl\| (\mc L_0 + \kappa) f  \bigr\|_{L^2}
	\leq \tfrac12 \bigl\| (\mc L_0 + \kappa) f  \bigr\|_{L^2} .
\end{equation*}
The former inequality proves \eqref{form comp}.  The latter shows that
\begin{equation*}
\tfrac14 \bigl\| (\mc L_0 + \kappa) f  \bigr\|^2_{L^2}\leq \bigl\| (\mc L_q + \kappa) f  \bigr\|^2_{L^2}\leq \tfrac94 \bigl\| (\mc L_0 + \kappa) f  \bigr\|^2_{L^2},
\end{equation*}
which proves the $s=1$ case of \eqref{equiv norms}.

Noting that $s=0$ is trivial, \eqref{equiv norms} follows for all $s\in[0,1]$ by (real or complex) interpolation.  Negative values of $s$ then follow by duality.

To prove that $q\mapsto \mc L_q$ is norm-resolvent continuous, we first choose $q_n\to q$ in $L^2_+$ and then choose $\kappa$ large enough so that \eqref{equiv norms} holds for $q$ and for all $q_n$.  In particular, $\kappa$ belongs to the resolvent set of $\mc L_q$ and all $\mc L_{q_n}$. By the resolvent identity, \eqref{equiv norms}, and \eqref{rc},
\begin{align*}
\|R(\kappa, q_n) - R(\kappa,q)\|_{\op} &\lesssim \|R(\kappa,q_n) [q_nC_+\ol q_n - qC_+\ol q]R(\kappa, q)\|_{\op}\\
&\lesssim \|R_0(\kappa) [q_nC_+\ol q_n - qC_+\ol q]R_0(\kappa)\|_{\op}\\
&\lesssim \|q_n-q\|_{L^2}(\|q_n\|_{L^2} + \|q\|_{L^2})\|R_0(\kappa)\|_\op,
\end{align*}
which converges to zero as $n\to \infty$.
\end{proof}

The operator $\mc P(q)$ from the Lax pair \eqref{Lax pair} plays a less central role in our analysis than $\mc L_q$.  Our next proposition constructs a unitary transformation built from the associated differential equation \eqref{Pode}.  It will suffice for us to work with smooth solutions $q(t)$ and so avoid any functional-analytic subtleties of the type addressed in the preceding proposition.

\begin{proposition}\label{P:Pflow}
Let $q(t)$ be a global $H^\infty_+(\R)$ solution of \eqref{CM-DNLS}.  For all $t_0\in\R$ and all  $\psi_0\in L^2_+(\R)$, the initial value problem
\begin{align}\label{Pode}
\tfrac{d }{dt} \psi(t) = \mc P(q(t)) \psi(t) \qtq{with} \psi(t_0)=\psi_0
\end{align}
admits a unique $C_tL^2_+\cap C_t^1 H^{-2}_+$ solution; this is global in time.   Moreover, for each $t\in\R$ the mapping $U(t;t_0):\psi_0\mapsto\psi(t)$ is unitary on $L^2_+$,
\begin{gather}\label{L conj}
q(t)=U(t;t_0)q(t_0), \qtq{and} \mc L(q(t)) = U(t;t_0) \mc L(q(t_0)) U^*(t;t_0).
\end{gather}
If $\psi_0\in H^\infty_+(\R)$, then so too is $\psi(t)$ for all $t\in \R$.  Finally, if $\psi_0\in H^\infty_+(\R)$, $\langle x\rangle \psi_0\in L^2(\R)$, and $\langle x\rangle q(0)\in L^2(\R)$, then $\langle x\rangle \psi(t)\in L^2(\R)$ for all $t\in \R$. 
\end{proposition}

\begin{proof}
Given $\eps>0$, we consider the regularized initial value problem
\begin{align}\label{Pode eps}
\tfrac{d }{dt} \psi_\eps(t) = i\partial^2 \psi_\eps(t)  \pm 2q(t) \tfrac{\partial}{1-\eps^2\partial^2} C_+\ol{q}(t) \psi_\eps(t) \qtq{with} \psi_\eps(t_0)=\psi_0 .
\end{align}
This is readily solved locally in time by employing Duhamel's formula,
\begin{align}\label{Pode eps'}
\psi_\eps(t) = e^{i(t-t_0)\Delta} \psi_\eps(t_0) \pm 2 \int_{t_0}^t e^{i(t-s)\Delta}  q(s) \tfrac{\partial}{1-\eps^2\partial^2} C_+\ol{q}(s) \psi_\eps(s)\,ds,
\end{align}
and contraction mapping in $C_t L^2_+$.

For any such solution, $e^{-i(t-t_0)\Delta} \psi_\eps(t)$ belongs to $C^1_t L^2_+$ and so one may readily verify that the $L^2$ norm is conserved.  This in turn yields global well-posedness by iterated contraction mapping. 

The $L^2$ conservation law also implies that the flow maps $U_\eps(t;t_0)$ are isometries.  Moreover, the uniqueness of solutions (guaranteed by contraction mapping) shows that $U_\eps(t;t_0)$ and $U_\eps(t_0;t)$ are inverses of one another and so both are unitary.

Proceeding by induction, taking spatial derivatives of \eqref{Pode eps'} and applying the Gronwall inequality, we find that for each $n\in \N$ and each $T>0$,
\begin{align}\label{1066}
 \| \psi_\eps(t) \|_{H^n_+} \leq C_n(q,T) \| \psi_0 \|_{H^n_+}  \qtq{for all} t,t_0\in [-T,T] .
\end{align}
Note that the constant $C_n(q,T)$ does not depend on $\eps>0$.  

Let us now compare differing regularizations, beginning with the case $\psi_0\in H^2_+$. Given $0<\eps<\eta$, taking the time derivative and integrating by parts to exploit the antisymmetry of the regularization, we obtain
\begin{align*}
\tfrac{d}{dt}\| \psi_\eta(t) - \psi_\eps&(t)\|_{L^2}^2 \\ &\lesssim \|\psi_\eta(t) - \psi_\eps(t) \|_{L^2} 
 	\bigl\| q(t) \tfrac{(\eta^2 -\eps^2)\partial^3}{(1-\eta^2\partial^2)(1-\eps^2\partial^2)} C_+\ol q(t)[\psi_\eta(t) + \psi_\eps(t)] \bigr\|_{L^2} \\
&\lesssim \eta \|\psi_\eta(t) - \psi_\eps(t) \|_{L^2} \| q(t) \|_{H^2}^2
 	\bigl\| \psi_\eta(t) + \psi_\eps(t) \bigr\|_{H^2}  .
\end{align*}
Combining this, \eqref{1066}, and Gronwall's inequality, we deduce that for each $T>0$,
\begin{align*}
 \| \psi_\eta(t) - \psi_\eps(t) \|_{L^2}^2 &\lesssim  \eta C( q , T ) \bnorm{ \psi_0 }_{H^2}^2 \qtq{for all} t,t_0\in [-T,T] .
\end{align*}
This proves that  $\psi_\eps(t)$ converges in $C_t([-T,T];L^2_+)$ as $\eps\to 0$ for any choice of $T>0$ and any $\psi_0\in H^2_+$.

Unitarity of $U_\eps(t;t_0)$ together with the density of $H^2$ show that this convergence carries over to every $T>0$ and every $\psi_0\in L^2_+$.  Note that the limiting functions $\psi(t)$ satisfy
\begin{align}\label{Pode'}
\psi(t) = e^{i(t-t_0)\Delta} \psi(t_0) \pm 2 \int_{t_0}^t e^{i(t-s)\Delta}  q(s) \partial C_+\ol{q}(s) \psi(s)\,ds,
\end{align}
and so are $C_tL^2_+\cap C^1_t H^{-2}_+$ solutions to \eqref{Pode}; moreover, these solutions inherit conservation of $L^2$ and the persistence of regularity estimate \eqref{1066}.
 
Earlier we saw that $U_\eps(t;t_0) U_\eps(t_0;t) = \Id$.  Sending $\eps\to 0$, we deduce that $U(t;t_0)$ is unitary for all $t,t_0$.

Suppose (toward a contradiction) that \eqref{Pode} admits two $C_t L^2_+$ solutions $\psi(t)$ and $\widetilde \psi(t)$ that differ at some time $t_1\in\R$. Evidently, there is a choice of $\phi_1\in H^\infty_+$ so that $\langle \phi_1, \psi(t_1) - \widetilde \psi(t_1)\rangle \neq 0$.  By the analysis presented above, we can find a $C_t H^\infty_+$ solution $\phi(t)=U(t;t_1)\phi_1$. Integrating by parts (to exhibit the antisymmetry of $\mc P$),  it follows that $\frac{d}{dt} \langle \phi(t),\,\psi(t)-\widetilde\psi(t)\rangle = 0$.  This yields a contradiction: 
$$
0 \neq \langle \phi_1, \psi(t_1) - \widetilde \psi(t_1)\rangle = \langle \phi(t_0),\,\psi(t_0)-\widetilde\psi(t_0)\rangle =0.
$$

We have now completed the proof of existence, uniqueness, and persistence of regularity for solutions to \eqref{Pode}, and unitarity of the associated propagator.

We now turn to \eqref{L conj}.  As noted in the introduction, one of the favorable attributes of our choice of $\mc P$ is that $\frac{d}{dt} q = \mc P(q) q$. The uniqueness of such solutions demonstrated above shows that $q(t) = U(t;t_0) q(t_0)$.  

For any $\phi_0\in H^\infty_+(\R)$, the fact that $\mc L$ and $\mc P$ form a Lax pair then shows that
\begin{align*}
t\mapsto \bigl[ \mc L(q(t)) U(t;t_0) - U(t;t_0) \mc L(q(t_0))\bigr]  \phi_0
\end{align*}
constitutes a solution to \eqref{Pode} with zero initial data at time $t=t_0$.  Uniqueness of such solutions combined with unitarity of $U(t;t_0)$ then proves \eqref{L conj}.

It remains to discuss weighted estimates.  We begin with $q(t)$ itself and assume that $\langle x\rangle q^0(x)\in L^2(\R)$.  Integration by parts shows that
\begin{align}\label{1066 xq}
 \frac{d}{dt} \int w(t) |q(t,x)|^2\,dx = \int w'(x) \bigl[ 2\Im(\ol q q') \mp |q|^4\bigr](t,x)\,dx
\end{align}
for any smooth and bounded weight $w$.  Observe that the rate of growth is mediated by the size of the derivative of $w$, as well as the assumed $H^1$ bounds on $q(t)$.  In this way, using a sequence of bounded approximating weights and a Grownwall argument, we obtain a quantitative bound on the $L^2(\R)$ norm of $\langle x\rangle q(t,x)$. 

The weighted norms of $\psi(t)$ do not follow such a simple evolution as \eqref{1066 xq}.  Nevertheless, proceeding in the same vein, we verify that
\begin{align}\label{1066 xpsi}
 \Bigl| \tfrac{d}{dt} \bigl\|\langle x\rangle \psi\bigr\|_{L^2}^2 \Bigr|
 	\lesssim \|\langle x\rangle \psi\|_{L^2} \|\psi\|_{H^1}
		+ \|\langle x\rangle q \|_{L^2} \|\langle x\rangle \psi\|_{L^2}\| q\|_{H^1}\|\psi\|_{H^1}
\end{align}
and then deduce that $\langle x\rangle \psi(t,x)\in L^2(\R)$ for all time via Gronwall.
\end{proof}

It is not difficult to see that the arguments just presented also prove an analogue of Proposition~\ref{P:Pflow} in the case that $q(t)$ is merely a local-in-time smooth solution.  We do not need such generality because smooth solutions with mass below the equicontinuity threshold are automatically global:

\begin{lemma}\label{L:horizon}
Let $q:[0,T) \to H^\infty_+(\R)$ be a solution to \eqref{CM-DNLS}.  If $\{q(t): t\in [0,T)\}$ is $L^2$-equicontinuous, then $q(t)$ extends $H^\infty_+(\R)$-continuously to $[0,T]$. 
\end{lemma}

\begin{proof}
We will use conservation of the Hamiltonian to show that $q(t)$ is $H^1$-bounded on $[0,T)$.  This guarantees $H^1$-continuous extension by the results of \cite{GerardLenzmann} described in Theorem~\ref{T:GL}.  It also provides the base step of their inductive argument yielding control over all higher Sobolev norms.

Given a frequency cutoff parameter $N$, we may combine the Bernstein and Gagliardo--Nirenberg inequalities to see that
\begin{align*}
\| q\|_{L^6}^6 \lesssim N^2 \| q_{\leq N} \|_{L^2}^6 + \| q_{>N}\|_{L^2}^4 \| q_{>N}' \|_{L^2}^2
	\lesssim N^2 M(q)^3 + \| q_{>N}\|_{L^2}^4 \| q' \|_{L^2}^2
\end{align*}
uniformly for $0\leq t <T$.   As $q(t)$ is assumed to be $L^2$-equicontinuous, given $\eta>0$, we may choose $N(\eta)$ so that
\begin{align}\label{h1}
   \| q(t) \|_{L^6}^6 \lesssim \eta^2 \| q'(t)  \|_{L^2}^2 +  N(\eta)^2 M(q) \quad\text{uniformly for $t\in[0,T)$.}
\end{align}
The role of this estimate will be to control the discrepancy between the full Hamiltonians \eqref{E:Hamil} and their quadratic part: $H_\pm^{[2]}:=\int \frac12 |q'|^2\, dx$.

Using $L^3$-boundedness of $C_+$, we find that for any $0<\eta<1$,
\begin{align*}
\Bigl| H_\pm(q(t)) - H_\pm^{[2]}(q(t))\Bigr| \lesssim \eta \|q'(t)\|_{L^2}^2 + \eta^{-1}  \| q(t)\|_{L^6}^6 .
\end{align*}
Combining this with \eqref{h1} and choosing $\eta$ sufficiently small, we deduce that $q(t)$ is bounded in $H^1$ from the conservation of the Hamiltonian.
\end{proof}

Schatten-class properties of the resolvent of the Lax operator $\mc L_q$ will play a key role in proving equicontinuity of orbits in the next section.  To demonstrate these properties, we will employ the following information about the free resolvent:

\begin{lemma}\label{L:HS est}
For $q\in L^2_+(\R)$ and $\kappa >0$,
\begin{align}
\norm{ C_+ \ol q R_0(\kappa) }_{\hs}^2 = \norm{ R_0(\kappa)q }_{\hs}^2
	&= \frac{1}{2\pi} \int_0^\infty \frac{|\wh{q}(\xi)|^2}{\xi+\kappa}\,d\xi, \label{hs}\\
\norm{ C_+\ol{q}R_0(\kappa)q }_{\hs}&\lesssim \|q\|_{L^2}^2. \label{hs improved}
\end{align}
\end{lemma}

\begin{proof}
The two operators in \eqref{hs} are adjoints of one another.  Their Hilbert--Schmidt norm follows from an explicit computation in Fourier variables:
\begin{align*}
\norm{ R_0(\kappa)q }_{\hs}^2 = \tr\big\{ R_0(\kappa)qC_+\ol{q} R_0(\kappa) \big\}
&= \frac{1}{2\pi}  \int_0^\infty \!\! \int_0^\infty \!\!\frac{1}{(\xi+\kappa)^2}\widehat{q}(\xi-\eta)\widehat{\ol{q}}(\eta-\xi)\, d\eta\,d\xi\\
&= \frac{1}{2\pi} \!\!\int_0^\infty  \!\!\int_0^\infty \frac{|\wh{q}(\xi)|^2}{(\xi+\eta+\kappa)^2}\, d\eta\,d\xi\\
&= \frac{1}{2\pi} \!\!\int_0^\infty \frac{|\wh{q}(\xi)|^2}{\xi+\kappa}\,d\xi.
\end{align*}

We turn now to \eqref{hs improved}.  A simple frequency analysis shows that on $L^2_+$ we have the operator identity
$$
C_+\ol{q}R_0(\kappa)q = \sum_{N_1,N_2} \ \sum_{M\geq \frac1{10} N_1\vee N_2}C_+\smash[b]{\ol q}_{N_1}R_0(\kappa)P_M q_{N_2},
$$
where we employ the notation $ N_1\vee N_2 = \max\{ N_1, N_2\}$. Thus, using the H\"older inequality for Schatten classes together with \cite[Theorem~4.1]{Trace Ideals} and the Bernstein inequality, we may bound
\begin{align*}
\norm{ C_+\ol{q}R_0(\kappa)q }_{\hs}
&\lesssim \sum_{N_1,N_2}\sum_{M\geq \frac1{10} N_1\vee N_2} \norm{\ol{q}_{N_1}\sqrt{R_0(\kappa)}P_M }_{\I_4}\norm{P_M \sqrt{R_0(\kappa)} q_{N_2}}_{\I_4}\\
&\lesssim \sum_{N_1,N_2}\sum_{M\geq \frac1{10} N_1\vee N_2} \norm{q_{N_1}}_{L^4}\norm{q_{N_2}}_{L^4} \bigl\| \sqrt{\tfrac1{\xi+\kappa}}\bigr\|_{L^4(|\xi|\sim M)}^2\\
&\lesssim \sum_{N_1,N_2}\sum_{M\geq \frac1{10} N_1\vee N_2} (N_1 N_2)^{\frac14} \frac{M^{\frac12}}{M+\kappa} \norm{q_{N_1}}_{L^2}\norm{q_{N_2}}_{L^2}\\
&\lesssim \sum_{N_1,N_2}\frac{(N_1 N_2)^{\frac14}}{(N_1\vee N_2)^\frac12}  \norm{q_{N_1}}_{L^2}\norm{q_{N_2}}_{L^2}\\
&\lesssim \|q\|_{L^2}^2,
\end{align*}
where the last line follows from applying the Schur test to sum in $N_1$ and $N_2$.
\end{proof}

We adopt the following notation for the resolvent of $\mc L_q$:
\begin{equation}\label{R series}
R(\kappa,q) = (\mc L_q +\kappa)^{-1} .
\end{equation}
Although $R(\kappa)$ does not belong to trace class, the difference $R(\kappa)-R_0(\kappa)$ does, at least for $\kappa$ sufficiently large.  This is the topic of the next lemma.  More striking is the fact that this trace is conserved under the \eqref{CM-DNLS} flow; this is the subject of the subsequent Proposition~\ref{L:conserv law}.

\begin{lemma}\label{t:beta}
Given $q\in L^2_+(\R)$ and $\kappa>0$ such that $\mc L_q+\kappa$ is positive definite, 
\begin{equation}\label{loop}
R(\kappa)-R_0(\kappa), \quad R_0(\kappa)qC_+\ol q R_0(\kappa),  \qtq{and}  R_0(\kappa)qC_+\ol{q} R(\kappa) qC_+\ol{q} R_0(\kappa)   
\end{equation}
all belong to the trace class $\I_1.$ Moreover, for any bounded and equicontinuous $Q$,
\begin{gather}\label{beta est}
\lim_{\kappa\to\infty} \sup_{q\in Q} \kappa \big\| R_0(\kappa)qC_+\ol{q} R(\kappa) qC_+\ol{q} R_0(\kappa) \big\|_{\I_1} =0.
\end{gather}
\end{lemma}

\begin{proof}
The resolvent identity 
\begin{equation}\label{beta12}
R(\kappa)-R_0(\kappa) = \pm R_0(\kappa)qC_+\ol q R_0(\kappa) + R_0(\kappa)qC_+\ol{q} R(\kappa) qC_+\ol{q} R_0(\kappa)  
\end{equation}
reduces consideration of the first operator in \eqref{loop} to treatment of the other two. 

The second operator in \eqref{loop} has the form $TT^*$ with $T=R_0(\kappa)q$ and so is positive definite.  Together with \eqref{hs}, this yields
\begin{equation}
\big\| R_0(\kappa)qC_+\ol q R_0(\kappa) \big\|_{\I_1} = \tr\bigl\{ R_0(\kappa)qC_+\ol q R_0(\kappa) \bigr\}
	= \frac1{2\pi}\int_0^\infty \frac{|\wh{q}(\xi)|^2}{\xi+\kappa}\,d\xi.
\end{equation}
Thus our claims about this operator hold for all $\kappa>0$.

To treat the third operator in \eqref{loop}, we need some information about $R(\kappa)$.  If we choose $\kappa_0=\kappa_0(\{q\})$ according to Proposition~\ref{P:Lax}, then \eqref{form comp} guarantees that $0 \leq R(\kappa) \leq 2 R_0(\kappa)$ as quadratic forms whenever $\kappa\geq \kappa_0$.  In this way, we see that
\begin{align*}
0 \leq R_0(\kappa)qC_+\ol{q} R(\kappa) qC_+\ol{q} R_0(\kappa) \leq 2 \, R_0(\kappa)qC_+\ol{q} R_0(\kappa) qC_+\ol{q} R_0(\kappa),
\end{align*}
still in the sense of quadratic forms. From \eqref{hs} and \eqref{hs improved}, it follows that
\begin{equation}\label{2.23}\begin{aligned}
\big\| R_0(\kappa)qC_+\ol{q} R(\kappa) qC_+\ol{q} R_0(\kappa) \big\|_{\I_1} 
	&\leq 2\, \big\| R_0(\kappa)q \bigr\|_{\I_2}^2 \big\| C_+ \ol{q} R_0(\kappa)\,q \bigr\|_\op \\
	&\lesssim \kappa^{-1}\|q\|_{L^2}^4.
\end{aligned}\end{equation}

Claim \eqref{beta est} follows from \eqref{2.23}, \eqref{hs}, and \eqref{rc equi} after choosing $\kappa_0=\kappa_0(Q)$.
\end{proof}

\begin{lemma}[A conservation law]\label{L:conserv law}
Let $q(t)$ denote an $H^\infty_+(\R)$ solution to~\eqref{CM-DNLS} and suppose $\kappa>0$ is such that $\mc L_{q(0)}+\kappa$ is positive definite.  Then $\mc L_{q(t)}+\kappa$ is positive definite for all times $t$ and
\begin{equation}\label{beta consv}
\tfrac{d}{dt}\tr\{R(\kappa, q(t))-R_0(\kappa)\} = 0 .
\end{equation}
\end{lemma}

\begin{proof}
From \eqref{L conj} we see that $\mc L_{q(t)}+\vk$ is positive definite for all $t$ and all $\vk\geq\kappa$.  Together with Lemma~\ref{t:beta}, this ensures that the trace \eqref{beta consv} is well defined.

We claim that both  $\widetilde{\mc P}(q(t)) R(\kappa, q(t))^2$ and $R(\kappa, q(t))^2\widetilde{\mc P}(q(t))$ are trace class operators, where $\widetilde{\mc P}(q(t))$ is the alternative form of the Lax operator presented in \eqref{bad P}.  As a first step to verifying this, we observe that for any $f,g \in H^\infty(\R)$,
\begin{align*}
C_+ f C_+ g R_0(\kappa)^2 &= C_+f R_0(\kappa) \cdot (\mc L_0 + \kappa) C_+ g R_0(\kappa)^2 \\
&=C_+f R_0(\kappa) \Bigl[ C_+ g R_0(\kappa) - i  C_+ g' R_0(\kappa)^2 \Bigr]
\end{align*}
and so by \eqref{hs}, this operator is trace class.  Combining this with Lemma~\ref{t:beta}, we deduce that $C_+ f C_+ g R(\kappa)^2$ is also trace class.  This in turn proves that the operator $\widetilde{\mc P}(q(t)) R(\kappa, q(t))^2$ is trace class.  The fact that the adjoint of this operator is also trace class settles the analogous question for $R(\kappa, q(t))^2\widetilde{\mc P}(q(t))$.

As both products are trace class, it follows that
$$
\tr \bigl\{[\widetilde{\mc P}(q(t)), R(\kappa, q(t))^2]\bigr\}=0.
$$
By the Fundamental Theorem of Calculus, this yields 
$$
\tr \bigl\{R(\kappa, q(t))^2 - R(\kappa, q(0))^2\bigr\}=0.
$$
Integrating the spectral parameter over the interval $[\kappa, \varkappa]$ and rearranging, we get
\begin{align*}
\tr \bigl\{[R(\kappa, q(t)) - R_0(\kappa)] &-[R(\varkappa, q(t)) - R_0(\varkappa)] \bigr\}\\
&= \tr \bigl\{[R(\kappa, q(0)) - R_0(\kappa)] -[R(\varkappa, q(0)) - R_0(\varkappa)] \bigr\}.
\end{align*}
Using \eqref{beta12}, \eqref{hs}, and \eqref{beta est}, we see that uniformly in $t$,
\begin{align*}
 \bigl\| R(\varkappa, q(t)) - R_0(\varkappa) \bigr\|_{\I_1} = O(\vk^{-1}).
\end{align*}
Thus sending $\varkappa\to \infty$ we obtain the desired conservation law.
\end{proof}

\section{Equicontinuity in $L^2$}\label{S:3}

In this section we prove $L^2$-equicontinuity of orbits of $H^\infty_+$ solutions to \eqref{CM-DNLS}, both in the focusing and defocusing settings.  In the focusing case, we will require that the initial data satisfies \eqref{M<M*}.

Our key quantity for detecting equicontinuity in $L^2(\R)$ is
\begin{equation}\label{beta defn'}
\beta(\kappa,q) := M(q) \mp 2\pi\kappa \tr\{ R(\kappa)-R_0(\kappa) \}  .
\end{equation}
In view of Lemma~\ref{L:conserv law}, this quantity is conserved by $H^\infty_+$ solutions to \eqref{CM-DNLS}.  Moreover, using \eqref{hs} we find that the quadratic term is given by
\begin{equation}\label{beta0 quadratic 2}
\beta\sbrack{2}(\kappa,q) = M(q) -2\pi\kappa \tr\{R_0(\kappa)qC_+\ol{q}R_0(\kappa)\}= \int_0^\infty \frac{\xi}{\xi+\kappa} |\wh{q}(\xi)|^2\,d\xi .
\end{equation}
By the resolvent identity \eqref{beta12}, the remainder can be written
\begin{equation}\label{beta0 higher}
\beta(\kappa,q) - \beta\sbrack{2}(\kappa,q) = \mp 2\pi\kappa \tr\{R_0(\kappa)qC_+\ol qR(\kappa) qC_+\ol qR_0(\kappa)\}. 
\end{equation}

In the defocusing case, it will suffice to exploit the direct connection between $\beta\sbrack{2}$ and $\wh q$ visible in \eqref{beta0 quadratic 2}.  In the focusing case, the remainder \eqref{beta0 higher} is more troublesome and for this purpose we introduce $\mc K$ norms inspired by \cite{MR4628747}:  Given an infinite subset $\mathcal K\subseteq 2^{\mathbb N}$, we define 
\begin{align}\label{K defn}
    \|q\|_{\mathcal K}^2&:=\|q\|_{L^2}^2 + \sum_{\kappa\in \mathcal K} \beta^{[2]}(\kappa,q).
\end{align}

The connections between $\beta\sbrack{2}(\kappa,q)$, $\mc K$ norms, equicontinuity, and $\beta(\kappa,q)$ can be summarized thus:

\begin{lemma}\label{L:equi via K}
Let $Q$ be a bounded subset of $L^2_+(\R)$.  The following are equivalent:
\begin{enumerate}
\item $Q$ is equicontinuous in $L^2(\R)$.
\item $\sup_{q\in Q}\beta\sbrack{2}(\kappa,q)  \to 0$ as $\kappa\to\infty$.
\item There exists an infinite set $\mathcal K\subseteq 2^{\mathbb N}$ so that $\sup_{q\in Q}\|q\|_{\mathcal K}<\infty$.
\end{enumerate}
Moreover, if $Q$ is $L^2$-equicontinuous, then
\begin{align}\label{equi to beta}
     \sup_{q\in Q} \; \beta(\kappa,q)  \to 0 \qtq{as} \kappa\to\infty.
\end{align}
\end{lemma}

\begin{proof}
To see that (1) and (2) are equivalent, we note that 
for any $0<\eta\leq 1$,
\begin{equation}
\norm{q_{>\kappa}}_{L^2}^2 \lesssim \beta\sbrack{2}(\kappa,q) \lesssim \eta \norm{q}_{L^2}^2 + \norm{q_{>\eta\kappa}}_{L^2}^2 .
\label{beta0 quadratic}
\end{equation}
Clearly, (2) implies (3).  That (3) implies (1) follows from the observation that 
\begin{align*}
    \|q\|_{\mathcal K}^2 &\gtrsim \|q\|_{L^2}^2 + \sum_{N\in 2^{\mathbb N}} \#\{\kappa\in \mathcal K: \kappa< N\}\;\|q_N\|_{L^2}^2.
\end{align*}

In order to deduce \eqref{equi to beta} from property (2), we must show that \eqref{beta0 higher} converges to zero as $\kappa\to \infty$ uniformly for $q\in Q$.  This follows from \eqref{beta est}.
\end{proof}

\begin{theorem}[Defocusing case]\label{T:defoc equi}
If $Q\subset H^\infty_+(\R)$ is bounded and equicontinuous in $L^2(\R)$, then the totality of states $Q^*=\{ e^{tJ\nabla\ham_-} q : q\in Q,\ t\in\R \} $ reached under the defocusing~\eqref{CM-DNLS} flow is also bounded and equicontinuous in $L^2(\R)$.
\end{theorem}

\begin{proof}
By Proposition~\ref{P:Lax}, we may chose $\kappa_0=\kappa_0(Q)$ so that $\mc L_q+\kappa$ is positive definite for all $q\in Q$ and $\kappa\geq \kappa_0$.  Moreover, as noted in Lemma~\ref{L:conserv law}, this property remains true at all later times and $R(\kappa,q(t))$ is positive definite.  As we are in the defocusing case, \eqref{beta0 higher} then shows that $\beta\sbrack{2}(\kappa,q(t)) \leq \beta(\kappa,q(t))$.  Combining this with the conservation law \eqref{beta consv}, we deduce that
\begin{equation}
\sup_{q\in Q^*} \beta\sbrack{2}(\kappa,q) 
\leq \sup_{q\in Q^*} \beta(\kappa,q) 
= \sup_{q\in Q} \beta(\kappa,q) .
\label{beta defocus}
\end{equation}
By \eqref{equi to beta}, the right-hand side converges to zero as $\kappa\to\infty$.  Therefore, we conclude that the left-hand side converges to zero as $\kappa\to\infty$, which in view of Lemma~\ref{L:equi via K} yields the equicontinuity of the set $Q^*$.\end{proof}

The argument just presented does not adapt to the focusing case because the remainder \eqref{beta0 higher} has an unfavorable sign.

\begin{theorem}[Focusing case]\label{T:foc equi}
If $Q\subset H^\infty_+(\R)$ is bounded and equicontinuous in $L^2(\R)$ and satisfies \eqref{M<M*}, then the set of orbits $Q^*=\{ e^{tJ\nabla\ham_+} q : q\in Q,\ t\in\R\}$ reached under the focusing~\eqref{CM-DNLS} flow is bounded and equicontinuous in $L^2(\R)$.
\end{theorem}

\begin{proof}
Let
\begin{equation*}
M := \sup \{ \norm{q}_{L^2}^2 : q\in Q \}.
\end{equation*}
Under our hypotheses, $M<2\pi$.  As the $L^2$ norm is conserved by solutions to \eqref{CM-DNLS}, we also have that
\begin{align}\label{mass bdd}
\sup_{q\in Q^*} \norm{q}_{L^2}^2\leq M<2\pi.
\end{align}

In view of \eqref{mass bdd} and \eqref{op est 4},
\begin{align*}
\bigl\langle f, q C_+ (\ol q f)\bigr\rangle  = \| C_+ (\ol q f) \|^2_{L^2}
	\leq \tfrac{M}{2\pi} \langle f, \mc L_0 f \rangle,
\end{align*}
for every $q\in Q^*$.  This implies
\begin{align}\label{777}
\mc L_q \geq \bigl(1-\tfrac{M}{2\pi}\bigr) \mc L_0   \qtq{and so also }  R(\kappa, q) \leq \tfrac{2\pi}{2\pi - M}  R_0(\kappa)
\end{align}
for any $\kappa>0$ and $q\in Q^*$. 

As $Q$ is equicontinuous in $L^2(\R)$, by Lemma~\ref{L:equi via K} we may find an infinite set $\mathcal K\subseteq 2^{\mathbb N}$ so that
\begin{align}\label{10:07}
\sup_{q\in Q}\|q\|_{\mathcal K}<\infty.
\end{align}

Using \eqref{beta0 higher} and \eqref{777}, we may bound
\begin{align*}
\sum_{\kappa\in \mathcal K} \bigl| \beta(\kappa,q)-\beta\sbrack{2}(\kappa,q) \bigr| 
\lesssim_M  \sum_{\kappa\in \mathcal K} \kappa \tr\bigl\{R_0(\kappa)q C_+ \ol{q}R_0(\kappa)qC_+\ol{q}R_0(\kappa)\bigr\}.
\end{align*}

To continue, we note that for $\kappa\geq 1$ we have $R_0(\kappa)\leq R_0(1)$ as operators on $L^2_+$ and so
$$
R_0(\kappa)q C_+ \ol{q}R_0(\kappa)qC_+\ol{q}R_0(\kappa)\leq R_0(\kappa)q C_+ \ol{q}R_0(1)qC_+\ol{q}R_0(\kappa).
$$
Thus, cycling the trace we obtain
\begin{align*}
\sum_{\kappa\in \mathcal K} \bigl| \beta(\kappa,q)-\beta\sbrack{2}(\kappa,q) \bigr| 
&\lesssim   \tr\bigl\{\sqrt{R_0(1)} qC_+\ol{q}\sum_{\kappa\in \mathcal K} \kappa R_0(\kappa)^2 q C_+ \ol{q} \sqrt{R_0(1)}\bigr\}.
\end{align*}
Observing that $\sum_{\kappa\in \mathcal K} \kappa R_0(\kappa)^2 \leq R_0(1)$ and invoking \eqref{hs improved}, we deduce that
\begin{align*}
\sum_{\kappa\in \mathcal K} \bigl| \beta(\kappa,q)-\beta\sbrack{2}(\kappa,q) \bigr| 
&\lesssim   \tr\bigl\{\sqrt{R_0(1)} qC_+\ol{q}R_0(1) q C_+ \ol{q} \sqrt{R_0(1)}\bigr\}\\
&\lesssim \norm{C_+\ol{q} R_0(1) q}_{\hs}^2 \lesssim \|q\|_{L^2}^4 \lesssim M^2,
\end{align*}
uniformly for $q\in Q^*$.  Using the conservation of $\beta$ (cf. Lemma~\ref{L:conserv law}) and \eqref{10:07}, we conclude that 
\begin{align*}
\sup_{q\in Q^*} \sum_{\kappa\in \mathcal K} \beta\sbrack{2}(\kappa,q)  \lesssim \sup_{q\in Q} \sum_{\kappa\in \mathcal K} \beta\sbrack{2}(\kappa,q)  + M^2<\infty,
\end{align*}
which in view of Lemma~\ref{L:equi via K} establishes the equicontinuity of $Q^*$ in the $L^2(\R)$ topology.
\end{proof}

\section{Explicit formula}\label{S:4}

The climax of this section is a proof of the explicit formula \eqref{magic formula} in a restricted case:

\begin{theorem}\label{T:smooth EF}
For any  $H^\infty_+(\R)$ solution $q(t)$ to~\eqref{CM-DNLS} with initial data satisfying $\langle x\rangle q^0 \in L^2(\R)$,
\begin{equation}\label{smooth EF}
q(t,z) = \tfrac{1}{2\pi i} I_+\big( (X + 2t\mc{L}_{q^0} - z)^{-1} q^0 \bigr)
\end{equation}
for all $z$ with $\Im z>0$.
\end{theorem}

Ultimately, we will prove Theorem~\ref{t:magic formula}, which covers all $L^2_+(\R)$ solutions by taking a limit of smooth solutions. As a foundation for this, we will work to understand the basic objects of the explicit formula in the full generality of $q\in L^2_+(\R)$, rather than just the case $q\in H^\infty_+(\R)$ treated in Theorem~\ref{T:smooth EF}.

We begin with a few necessary preliminaries.  Our first lemma represents the key outgrowth of the fact that $\mc P1 = 0$.  As the constant function $1$ does not belong to the natural domain of $\mc P$, we use the sequence of approximations $\chi_y(x):=\tfrac{iy}{x+iy}$ which converge to unity (uniformly on compact sets) as $y\to\infty$.

\begin{lemma}\label{t:P1=0}
Let $q(t)$ be a global $H^\infty_+(\R)$ solution of \eqref{CM-DNLS} and let $U(t;t_0)$ denote the associate family of unitary operators constructed in Proposition~\ref{P:Pflow}.  Then 
\begin{gather}\label{U chi} 
\limsup_{y\to\infty}   \| U^*(t;0)\chi_y - \chi_y\|_{L^2} = 0 
\end{gather}
for any time $t\in\R$.  Recall that $\chi_y(x) := \tfrac{iy}{x+iy}$.
\end{lemma}

\begin{proof}
Given $t\in\R$, let $Q=\{q(s):|s|\leq |t|\}$, which is evidently precompact in $H^1_+(\R)$.  We first claim that
\begin{equation}\label{E:P1=0}
\lim_{y\to\infty} \bnorm{ \mc P(q) \chi_y }_{L^2} = 0\quad\text{uniformly for $q\in Q$.}
\end{equation}
To see this, we write 
\begin{equation*}
\mc{P} \chi_y = i\chi''_y \pm 2qC_+\big( \ol{q}'(\chi_y-1) \big) \pm 2qC_+(\ol{q}\chi'_y) ,
\end{equation*}
and use~\eqref{op est 4} to estimate
\begin{equation*}
\norm{\mc P\chi_y}_{L^2} \lesssim \snorm{\chi''_y}_{L^2} + \norm{q}_{L^\infty} \norm{\ol{q}'(\chi_y-1)}_{L^2} + \norm{q}_{L^\infty} \norm{q}_{L^2} \snorm{ \chi'_y}_{L^\infty} .
\end{equation*}
As $Q\subset H^1_+(\R)$ is precompact, we deduce \eqref{E:P1=0}.

Turning now to \eqref{U chi}, we note that $\chi_y\in H^\infty_+(\R)$ and so
\begin{align*}
\tfrac{d}{dt}  \langle \chi_y, U(t;0)\psi_0 \rangle 
	= - \langle \mc P(q(t)) \chi_y, U(t;0)\psi_0 \rangle
\end{align*}
for any $\psi_0\in L^2_+$.  Thus, by the Fundamental Theorem of Calculus, 
\begin{align*}
\bigl\| U(t;0)^*\chi_y-\chi_y\bigr\|_{L^2} \leq \int_{-|t|}^{|t|} \| \mc P(q(s)) \chi_y\|_{L^2}\,ds , 
\end{align*}
which converges to zero as $y\to\infty$ by \eqref{E:P1=0}.
\end{proof}

As noted above, $\chi_y(x)\to 1$ as $y\to\infty$, uniformly on compact sets.  Correspondingly, one may view $\langle \chi_y, q\rangle$ as an approximation for the total integral of $q$.  Building on this idea, we now define the \emph{unbounded} linear functional $I_+$ appearing in the \eqref{smooth EF}.  As we will see from the equivalent representations in \eqref{I+ def'}, this is closer to twice the integral of $q$.

\begin{definition} The (unbounded) linear functional $I_+$ is defined by
\begin{equation}\label{I+ def}
I_+(f) := \lim_{y\to\infty} \sqrt{2\pi} \int_0^\infty y e^{-y\xi} \, \wh f(\xi) \,d\xi 
\end{equation}
with domain $D(I_+)$ given by the set of those $f\in L^2_+(\R)$ for which the limit exits.
\end{definition}

By computing the Fourier transform of $\chi_y$ and by noting that $\chi_y+\smash[b]{\ol\chi_y}$ is equal to $2\pi y$ times the Poisson kernel, we find the following equivalent definitions of $I_+$:
\begin{equation}\label{I+ def'}
I_+(f) =  \lim_{y\to\infty} \bigl\langle \chi_y, f \bigr\rangle =  \lim_{y\to\infty}  \int \tfrac{2y^2}{x^2+y^2} f(x) \,dx = \lim_{y\to\infty} 2\pi y f(iy) 
\end{equation}
and that $D(I_+)$ is comprised of those $f\in L^2_+(\R)$ for which these limits exist.

Let us now describe the operator $X$ appearing in \eqref{smooth EF}.  If we were working in $L^2(\R)$, then $f(x)\mapsto e^{-itx}f(x)$ defines a unitary group, the generator of which is multiplication by $x$. (The factor $-i$ appearing here is the usual convention from quantum mechanics.)  However, we are working in the Hardy space $L^2_+(\R)$ and for $t>0$, multiplication by $e^{-itx}$ does not preserve this space.  The natural analogue on the Hardy space $L^2_+(\R)$ is
\begin{equation}\label{X 1}
e^{-itX} f = C_+\bigl( e^{-itx} f \bigr) = \tfrac{1}{\sqrt{2\pi}} \int_0^\infty e^{i\xi x} \widehat f(\xi + t)\,d\xi ,
\end{equation}
and we define the operator $X$ as the generator of this semigroup:
\begin{equation*}
D(X) = \bigl\{ f\in L^2_+ (\R): \widehat f\in H^1\bigl([0,\infty)\bigr) \bigr\} \qtq{and} \widehat{Xf}(\xi) =i\tfrac{d{\widehat f}}{d\xi}(\xi)\qtq{for} f\in D(X).
\end{equation*}

The semigroup \eqref{X 1} is clearly a (strongly continuous) contraction semigroup.  The famous Hille--Yoshida Theorem identifies the generators of such semigroups; see \cite{EN,Reed1975}.  Unfortunately, the nomenclature for these generators is not fully settled: they may be \emph{accretive} or \emph{dissipative}, there is a sign ambiguity, as well as the choice of whether or not to include the imaginary unit $i$.  Evidently, one should select the convention that best describes the operators one is working with and for this reason, we adopt the following:

\begin{convent}
An operator $T$ on a Hilbert space will be called \emph{accretive} if
\begin{equation}\label{accretive}
\Im \langle f, Tf\rangle \leq 0 \qtq{for all}  f\in D(T).
\end{equation}
It is further \emph{maximally accretive} if it admits no proper accretive extension.
\end{convent}

In this way, $T$ is accretive in our sense if and only if $iT$ is accretive in the sense of \cite{Reed1975} if and only if $-iT$ is dissipative in the sense of \cite{EN}.  Moreover, the Hille--Yoshida Theorem then says that an operator $T$ is maximally accretive if and only if $e^{-itT}$ defines a contraction semigroup.  In particular, the operator $X$ introduced above is maximally accretive.

The spectrum of $X$ constitutes the closed lower half-plane.  For $\Im z>0$, the resolvent is given by
\begin{equation}\label{X 2}
(X-z)^{-1} f = \tfrac{f(x)-f(z)}{x-z}
\end{equation}
where $f(z)$ is defined by \eqref{PIF11}.  Together with the functional $I_+$, this gives rise to the following formulation of the Cauchy integral formula:
\begin{equation}\label{CIF}
f(z)  = \tfrac{1}{2\pi i} I_+\bigl( (X-z)^{-1} f \bigr) = \lim_{y\to\infty} \tfrac{1}{2\pi i} \bigl\langle \chi_y, (X-z)^{-1} f \bigr\rangle
\end{equation}
valid for all $f\in L^2_+(\R)$ and $\Im z>0$.

The commutator of $\mc P$ and the operator $X$ will be important for deriving the explicit formula \eqref{smooth EF}.  As a stepping stone, we record the commutator of $X$ with a generic Toeplitz operator $f\mapsto C_+(gf)$.  This may be obtained from straightforward computations in Fourier variables. It also arises naturally in the setting of the Benjamin--Ono equation, see \cite{Sun2021}*{Lem.~3.1}.  

\begin{lemma}\label{L:comm}
If $g\in H^{\infty}(\R)$ and $f\in D(X)$, then $C_+(gf)\in D(X)$ and
\begin{equation}\label{comm 1}
[X,C_+g] f = \tfrac{i}{2\pi} I_+(f) C_+ g.
\end{equation}
\end{lemma}

We are now ready to compute the commutator of $\mc P$ with $X$:

\begin{lemma}
If $q\in H^\infty_+(\R)$ and $\langle x\rangle q \in L^2(\R)$, then
\begin{equation}
[X,\mc P] = 2\mc L_q .
\label{comm 2}
\end{equation}
\end{lemma}
\begin{proof}
For $f\in L^2_+(\R)$ with $\langle x\rangle f\in H^2(\R)$, we use Lemma~\ref{L:comm} to compute
\begin{equation*}
[X,\mc P]f = -2i f'\pm \tfrac{i}{\pi}I_+\big\{ C_+(\ol{q}f)' \big\} q \mp 2qC_+(\ol{q}f) .
\end{equation*}
As derivatives vanish at zero frequency, we recognize that the right-hand side above is equal to $2\mc L_q f$.
\end{proof}

We now turn our attention to the operator $X+2t\mc L_q$, whose resolvent lies at the center of the explicit formula.  We begin with the case $q\equiv 0$.

As a a preliminary notion of $X+2t\mc L_0$, we observe that
\begin{align}\label{1513}
(Xf+2t\mc L_0f)\fwh(\xi) &= i \partial_\xi \wh{f}(\xi) + 2 t \xi \wh f(\xi) \\
&= e^{it\xi^2} i\partial_\xi \bigl[e^{-it\xi^2} \wh f(\xi)\bigr] \qtq{for all} f\in D(X)\cap D(\mc L_0). \notag
\end{align}
While it is evident that with this domain $X+2t\mc L_0$ is accretive, it is not maximally accretive.  This can be remedied by taking the closure:

\begin{lemma}
The closure of the naive sum \eqref{1513} is given by
\begin{equation}\label{A0 1}
[ (X+2t\mc L_0) f ] \fwh (\xi) = e^{it\xi^2} i\partial_\xi \big[ e^{-it\xi^2} \wh{f}(\xi) \big] ,
\end{equation}
with domain
\begin{equation}\label{A0 2}
D(X+2t\mc L_0) = \{ f \in L^2_+(\R) : e^{-it\xi^2}\wh{f} \in H^1([0,\infty)) \}.
\end{equation}
The operator $X+2t\mc L_0$ is maximally accretive, generating the contraction semigroup
\begin{equation}\label{BCH formula}
\big[ e^{-is(X+2t\mc L_0)} f \big]\fwh(\xi) = e^{is^2t} e^{-2ist(\xi +s)} \wh{f}(\xi+s),
\end{equation}
and is unitarily conjugated to $X$ by the free Schr\"odinger flow:
\begin{equation}\label{conj to X}
X+2t\mc L_0 =  e^{-it\Delta} X e^{it\Delta}  .
\end{equation}
\end{lemma}

\begin{proof}
It is perhaps easiest to work backwards.  One first observes that \eqref{BCH formula} defines a strongly continuous semigroup of contractions and then that the generator is the operator defined by \eqref{A0 1} and \eqref{A0 2} and satisfies \eqref{conj to X}.

As the last step, we check that this generator is indeed the closure of $X+2t\mc L_0$ when defined on $D(X)\cap D(\mc L_0)$.  As maximally accretive operators are closed, this amounts to the elementary task of verifying that our operator is contained in the closure of the naive sum.  
\end{proof}

With the proper meaning of $X+2t\mc L_0$ now set, we may turn our attention to its resolvent:

\begin{proposition}\label{t:A0}
For $t\in \R$ and $\Im z > 0$, 
\begin{equation}\label{A0 def}
A_0(t,z) := (X+2t\mc L_0-z)^{-1} \qtq{satisfies} \norm{A_0}_{\op} \leq \tfrac{1}{\Im z}.
\end{equation}
For any $t\neq 0$, we have
\begin{align}
\norm{A_0(t,z)}_{L^2_+ \to L^\infty_+} + \norm{A_0(t,z)}_{L^1_+ \to L^2_+} &\lesssim (|t| \Im z)^{-\frac12}, \label{A0 5}\\
\norm{A_0(t,z)}_{L^1_+ \to L^\infty_+} &\lesssim |t|^{-1}. \label{A0 4} 
\end{align}
The range of $A_0$ lies in the domain of $I_+$,
\begin{equation}
I_+\bigl( A_0(t,z) f \bigr)  = 2\pi i \big[ e^{it\Delta} f \big](z) \qtq{for all} t\in\R,\ \ \Im z>0,
\label{A0 6} 
\end{equation}
and the composition $I_+\circ A_0$ is bounded:
\begin{align}\label{Schrod Poisson}
\bigl| I_+\bigl( A_0(t,z) f \bigr) \bigr| \lesssim \min\bigl\{ |\Im z|^{-1/2} \|f\|_{L^2}, \ |t|^{-1/2} \| f\|_{L^1} \bigr\}.
\end{align}
\end{proposition}

\begin{proof}
The existence of the inverse \eqref{A0 def} for $\Im z>0$ and the associated norm bound are basic consequences of the fact that $X+2t\mc L_0$ is maximally accretive.

From~\eqref{conj to X} we see that
\begin{equation}
A_0(t,z) = e^{-it\Delta} (X-z)^{-1} e^{it\Delta} .
\label{A0 3}
\end{equation}
For $t\neq 0$, the free Schr\"odinger propagator is given by the explicit formula
\begin{equation}
\big[e^{it\Delta}f\big](x) = \tfrac{1}{\sqrt{4\pi i t}}\int e^{i(x-y)^2/4t}f(y)\,dy, 
\label{Schrod kernel}
\end{equation}
where we use the principal branch of the square root, so that the argument of $\sqrt{4\pi it}$ is $\tfrac{\pi}{4}\sgn(t)$.  From this, one easily derives the standard dispersive estimates
\begin{equation}
\bnorm{e^{it\Delta}}_{L^{\frac{p}{p-1}}\to L^p} \lesssim |t|^{\frac{1}{p}-\frac{1}{2}} \quad\text{for }2\leq p\leq \infty .
\label{Schrod dispersion}
\end{equation}
Combining this with \eqref{266}, we deduce
\begin{align}\label{Schrod Poisson 0}
\bigl|[e^{it\Delta}f](z)\bigr|\lesssim\min\bigl\{|t|^{-\frac12} \|f\|_{L^1} , |\Im z|^{-\frac12} \|f\|_{L^2} \bigr\}
\end{align}
for all $t\neq 0$ and $\Im z>0$.

We now turn to \eqref{A0 5}.  Using that $\mc F(\frac1{x-z}) = \sqrt{2\pi} i e^{-i\xi z} \chi_{\xi<0}$ and the Van der Corput Lemma in the form given in \cite[Corollary~VIII.1.2]{big Stein}, we may estimate
\begin{align}\label{2:31}
\bigl\|e^{-it\Delta} \bigl(\tfrac1{x-z}\bigr)\bigr\|_{L^\infty}\lesssim \Bigl\| \int_{-\infty}^0 e^{ix\xi+it\xi^2 -i\xi \Re z}e^{\xi \Im z}\,d\xi \Bigr\|_{L^\infty}\lesssim |t|^{-\frac12}.
\end{align}
Employing \eqref{A0 3}, \eqref{Schrod Poisson 0}, and~\eqref{X 2}, we may thus bound
\begin{align*}
\norm{A_0(t,z)f}_{L^\infty}
&\lesssim |t|^{-\frac12} \bigl\| \tfrac{(e^{it\Delta}f)(x)}{x-z}\bigr\|_{L^1} +\bigl\|e^{-it\Delta} \bigl(\tfrac1{x-z}\bigr)\bigr\|_{L^\infty} \bigl|(e^{it\Delta}f)(z)\bigr|\\
& \lesssim |t|^{-\frac12}\|e^{it\Delta}f\|_{L^2} \| (x-z)^{-1}\|_{L^2} +(|t| \Im z)^{-\frac12} \|f\|_{L^2}\\
&  \lesssim (|t| \Im z)^{-\frac12} \|f\|_{L^2}.
\end{align*}

Proceeding similarly, we obtain
\begin{align*}
\|A_0(t,z)f\|_{L^2}
&\lesssim \bigl\|(X-z)^{-1} e^{it\Delta} f\bigr\|_{L^2} \\
& \lesssim\| (x-z)^{-1}\|_{L^2} \bigl[\|e^{it\Delta}f\|_{L^\infty}+ \bigl|(e^{it\Delta}f)(z)\bigr|\bigr] \\
&  \lesssim (|t| \Im z)^{-\frac12} \|f\|_{L^1},
\end{align*}
which completes the proof of \eqref{A0 5}.

Employing \eqref{A0 3}, \eqref{Schrod kernel}, and~\eqref{X 2}, a straightforward computation yields
\begin{align*}
[A_0(t,z)f](x)&= \tfrac{i}{2|t| } \int_{\frac{y}t<\frac{x}t} e^{-\frac{i(x^2-y^2)}{4t} + \frac{i(x-y)z}{2t}} f(y)\, dy  - \bigl( e^{it\Delta}f\bigr)(z)\cdot \Bigl[e^{-it\Delta} \bigl(\tfrac1{\cdot-z}\bigr)\Bigr](x).
\end{align*}
Claim \eqref{A0 4} follows from this, \eqref{Schrod Poisson 0} and \eqref{2:31}.

By definition, the resolvent $A_0(t,z)$ carries $f\in L^2_+$ into the domain~\eqref{A0 2}, which is clearly a subset of $D(I_+)$.  The identity~\eqref{A0 6} follows immediately from~\eqref{A0 3} and~\eqref{CIF}.  Together with \eqref{Schrod Poisson 0}, this identity shows that $I_+\circ A_0$ satisfies \eqref{Schrod Poisson}.
\end{proof}

This completes our discussion of $X+2t \mc L_0$.   We turn now to the more difficult question of realizing $X+2t \mc L_q$ as a maximally accretive operator.  

If $q\in L^\infty\cap L^2_+$, then $qC_+\ol q$ is an $L^2$-bounded self-adjoint operator.  Under this hypothesis, $X+2t \mc L_q$ is naturally defined on \eqref{A0 2} and is maximally accretive; see, for example, \cite[\S III.1]{EN}.  In this way, we may define
\begin{align}\label{A for nice}
A(t,z;q) := (X + 2t\mc L_q - z)^{-1} \qtq{for all} t\in\R \qtq{and} \Im z >0,
\end{align}
whenever $q\in L^2_+(\R)\cap L^\infty(\R)$.

For general $q\in L^2_+(\R)$, the construction of this operator is not so simple.  Indeed, we find ourselves unable to apply the textbook theorems for perturbations of accretive operators, such as the analogue of the Kato--Rellich Theorem.  Rather, we will build the operator as a limit of operators with bounded potentials $q_n$ in concert with the abstract theory developed by Trotter \cite{Trotter} and Kato \cite{Kato}.  Our first step in this direction is the following lemma, which we will also need for the proof of Proposition~\ref{t:q formula}:

\begin{lemma}\label{L:882}
Fix $T>0$ and let $Q\subset L^2_+(\R)$ be bounded and equicontinuous.  For every $\eps>0$ there exists $b=b(\eps, Q)$ so that $\Im z\geq b$ implies both
\begin{equation}\label{1555}
\bnorm{ 2t \,C_+\ol{g}A_0(t,z) q }_{\op} \leq \eps \qtq{and}
	|t| \bnorm{ C_+ \ol q A_0(t,z) }_{\op} \bnorm{ A_0(t,z) q }_{\op} \leq \tfrac{\eps}{\Im z}
\end{equation}
uniformly for $|t|\leq T$ and $g,q \in Q$.
\end{lemma}

\begin{proof}
We begin with the first claim in \eqref{1555}, decomposing $q = q_{\leq N} + q_{>N}$ and similarly $g = g_{\leq N} + g_{>N}$, for a frequency cutoff $N\geq 1$ that will be chosen shortly.  In this way, \eqref{A0 4} yields
\begin{align}\label{1565}
2|t| \bnorm{ C_+\ol{g}A_0 q_{> N} &+ C_+\ol{g}_{>N} A_0 q_{\leq N} }_{\op} \notag\\
&\lesssim |t| \|A_0\|_{L^1_+\to L^\infty_+} \big[\norm{q_{> N}}_{L^2} \norm{g}_{L^2} +  \norm{q}_{L^2} \norm{g_{>N}}_{L^2}\big] \\
&\lesssim \norm{q_{> N}}_{L^2} \norm{g}_{L^2} +  \norm{q}_{L^2} \norm{g_{>N}}_{L^2}.  \notag
\end{align}
As $Q$ is bounded and equicontinuous, we may choose $N=N(Q)$ large to ensure that LHS\eqref{1565} is bounded by $\frac{\eps}{2}$.

For the remaining low-low frequency contribution, we use Bernstein's inequality to estimate
\begin{equation*}
\bnorm{ 2tC_+\ol{g}_{\leq N}A_0q_{\leq N} }_{\op} 
\leq 2|t| \norm{A_0}_{\op} \norm{g_{\leq N}}_{L^\infty}\norm{q_{\leq N}}_{L^\infty} 
\lesssim \tfrac{TN}{\Im z} \norm{g}_{L^2} \norm{q}_{L^2} .
\end{equation*}
As $Q$ is bounded and $N(Q)$ is finite, we may choose $b\geq 1$ sufficiently large so as to render this contribution smaller that $\frac\eps2$.

The treatment of the second claim in  \eqref{1555} follows a parallel path: 
\begin{align*}
\norm{ C_+ \ol q A_0 }_{\op} 
	&\lesssim \| q_{> N}\|_{L^2} \|A_0\|_{L^2_+\to L^\infty_+} + \| q_{\leq N}\|_{L^\infty} \|A_0\|_\op ,\\
\norm{ A_0 q }_{\op}
	&\lesssim \| q_{> N}\|_{L^2} \|A_0\|_{L^1_+\to L^2_+} + \| q_{\leq N}\|_{L^\infty} \|A_0\|_\op,
\end{align*}
and so by \eqref{A0 def} and \eqref{A0 5},
\begin{align*}
 |t| \cdot \norm{ C_+ \ol q A_0 }_{\op} \cdot \norm{ A_0 q }_{\op}  &\lesssim \tfrac{1}{\Im z}\Bigl[ \| q_{> N}\|_{L^2}^2 + \tfrac{NT}{\Im z} \| q\|_{L^2}^2 \Bigr] .
\end{align*}
As previously, we first choose $N=N(Q)$ and then $b(\eps, Q)$ to render this contribution acceptable.
\end{proof}

We are now prepared to construct the operator $X + 2t\mc L_q$ for $q\in L^2_+(\R)$, as well as its resolvent.  Further mapping properties of the resolvent will be elaborated in Proposition~\ref{t:A}.

\begin{proposition}\label{t:A'}
Fix $q\in L^2_+(\R)$.   For any $t\in\R$, $\Im z>0$, and any sequence $q_n\in L^2_+(\R)\cap L^\infty_+(\R)$ converging to $q$ in $L^2$-sense, the operators
\begin{equation}\label{916}
A_n(t,z) := (X + 2t\mc L_0  \mp 2t q_n C_+ \ol{q}_n - z)^{-1}
\end{equation}
converge in norm to an operator $A(t,z;q)$ that satisfies
\begin{equation}\label{918}
A(t,z;q) (X f + 2t\mc L_q f - zf) = f \qtq{whenever} f\in D(X)\cap D(\mc L_0).
\end{equation} 
Moreover, $z\mapsto A(t,z)$ is the resolvent of a maximally accretive operator, which we denote $X + 2t\mc L_q$.  Thus
\begin{equation}\label{920}
A(t,z;q) = (X + 2t\mc L_q - z)^{-1}   \qtq{satisfies} \norm{A(t,z;q)}_{\op} \leq \tfrac{1}{\Im z}
\end{equation}
for any $t\in\R$, $\Im z>0$, and any $q\in L^2_+(\R)$.
\end{proposition}

\begin{proof}
To treat the sequence $q_n$ and its limit $q$ in a parallel fashion, it is convenient to set $q_\infty=q$.   Convergence of the sequence guarantees that $Q:=\{q_n: 1\leq n\leq \infty \}$ is $L^2$ precompact and so bounded and equicontinuous.  By Lemma~\ref{L:882}, we may choose $b(\eps)$ so that
\begin{align}\label{930}
\sup_{1\leq n,m\leq\infty} \ \bnorm{ 2tC_+\smash[b]{\ol q}_nA_0(t,z) q_m }_{\op} \leq \eps \qtq{whenever} \Im z \geq b(\eps).
\end{align}
We particularly single out the value corresponding to $\eps=\tfrac12$, for which we write $b_0$.

For $\Im z \geq b_0$ and any $n\in \N$, the estimate \eqref{930} guarantees convergence of the resolvent series
\begin{equation}\label{936}
A_n(t,z) = A_0(t,z) \pm 2t A_0(t,z) q_n  \sum_{\ell \geq 0} \big[\pm 2t C_+ \smash[b]{\ol q}_n A_0(t,z) q_n \big]^{\ell} C_+\ol{q}_n A_0(t,z). 
\end{equation}
Note also that by \eqref{A0 5}, we have
\begin{equation}
\sup_{1\leq n\leq\infty} |t|^{\frac12} \| A_0(t,z) q_n \|_\op + |t|^{\frac12} \| C_+\smash[b]{\ol q}_n A_0(t,z) \|_\op  \lesssim_Q (\Im z)^{-\frac12} .
\end{equation}

It follows from the preceding discussion that for $\Im z \geq b_0$ we may define a bounded operator $A_\infty(t,z)$ by setting $n=\infty$ in \eqref{936} and, by a simple telescoping argument, we see that
\begin{equation}\label{944}
\| A_n(t,z) - A_\infty(t,z) \|_\op \to 0 \qtq{and so} \| A_\infty(t,z) \|_\op \leq \tfrac{1}{\Im z}
\end{equation}
for all $\Im z \geq b_0$ and $t\in\R$.

As the resolvents $A_n(t,z)$ converge at one point $z\in \C_+$, they must converge at all points in $\C_+$.  However, it is not true in general that the limit of resolvents of accretive operators is the resolvent of an accretive operator. Rather it is merely a pseudo-resolvent; see \cite{Kato}.  The kernel and range of a pseudo-resolvent are independent of the spectral parameter; however, it is only for a true resolvent that the kernel is trivial and the range is dense.

To verify that $A_\infty(z)$ is a true resolvent, we will show that 
\begin{equation}\label{952}
\lim_{y\to\infty} iy A_\infty(iy)f = -f \qtq{for all} f\in L^2_+(\R).
\end{equation}
This is always true for the resolvent of a maximally accretive operator (as one may see by integrating the semigroup); in particular, 
\begin{equation}\label{952'}
\lim_{y\to\infty} iy A_0(iy)f = -f \qtq{for all} f\in L^2_+(\R).
\end{equation}

By comparison, Lemma~\ref{L:882} shows that
\begin{align*}
\limsup_{y\to\infty}  y |t| \bigl\| A_0(iy) q \|_\op \bigl\| C_+ \ol q A_0(iy) \|_\op
 	\sum_{\ell \geq 0} \Big[|t| \bigl\| C_+ {\ol q} A_0(iy) q\bigr\|_\op \Big]^{\ell} =0
\end{align*}
for any fixed $q\in L^2_+(\R)$.  Thus \eqref{952} follows from \eqref{952'} and \eqref{936}.

We may now invoke the abstract Trotter--Kato theory; see \cite{EN}.  The resolvents of the maximally accretive operators $X+2t\mc L_{q_n}$ converge for all $z$ with $\Im z>0$. Moreover, \eqref{952} ensures that the limit $A_\infty(z)$ is injective with dense range.  It then follows that $A_\infty(z)$ is the resolvent of a densely defined operator.  By \eqref{944}, the operator is maximally accretive. 

Finally, we verify \eqref{918}, which illustrates the strong connection between the maximally accretive operator just constructed and the naive notion of the sum $X+2t \mc L_q$ defined on $D(X)\cap D(\mc L_0)$.  By definition of the resolvent, \eqref{918} holds if we replace $q$ by the bounded approximations $q_n$. To send $n\to\infty$, we first note that arguing as in the proof of Proposition~\ref{P:Lax} we have
\begin{align*}
\|C_+(\ol q f)\|_{L^\infty} \lesssim \|[C_+(\ol q f)]\fwh \|_{L^1}\lesssim \|\hlm \widehat q\|_{L^2} \|\langle \eta\rangle \widehat f\|_{L^2}\lesssim \|q\|_{L^2}\|f\|_{H^1}
\end{align*}
and consequently,
\begin{align*}
\bnorm{ [qC_+\ol q - q_n C_+\smash[b]{\ol q}_n]f  }_{L^2}
	&\lesssim \| q \|_{L^2} \| C_+[\ol q - \smash[b]{\ol q}_n ]f \|_{L^\infty} + \| q-q_n \|_{L^2} \| C_+\smash[b]{\ol q}_n f \|_{L^\infty} \\
&\lesssim \bigl[ \| q_n \|_{L^2} + \| q \|_{L^2} \bigr] \| q-q_n \|_{L^2} \| f\|_{H^1}.
\end{align*}
This converges to zero as $n\to \infty$, completing the proof of \eqref{918}.
\end{proof}

\begin{proposition}\label{t:A}
Given a bounded and equicontinuous set $Q\subset L^2_+(\R)$, the operators
\begin{equation}\label{A 1}
A(t,z;q) = (X + 2t\mc L_q - z)^{-1}
\end{equation}
introduced in Proposition~\ref{t:A'} satisfy 
\begin{gather}\label{A 4}
|t|^{\frac12}\norm{ A }_{L^2_+\to L^\infty_+} + |t|^{\frac12}\norm{ A }_{L^1_+\to L^2_+} \lesssim \tfrac{1}{\sqrt{\Im z}} + \tfrac{1}{\Im z}\\
|t| \norm{ A }_{L^1_+\to L^\infty_+} \lesssim 1 + \tfrac{1}{\Im z} \label{A 444}
\end{gather}
uniformly for $q\in Q$, $|t|\leq T$, and $\Im z >0$.  Likewise, for any $f\in L^2_+(\R)$,
\begin{gather}\label{1010}
\bigl| I_+\bigl[ A(t,z;q) f\bigr] \bigr| \lesssim \Bigl[\tfrac{1}{\sqrt{\Im z}} + \tfrac{1}{\Im z}\Bigr] \| f\|_{L^2}.
\end{gather}
\end{proposition}

\begin{proof}
In view of Lemma~\ref{L:882}, one may choose $b_0=b_0(Q)$ so that
\begin{equation}
\sup_{q\in Q} \bnorm{ 2tC_+\ol{q}A_0(t,z) q }_{\op} \leq \tfrac{1}{2} \qtq{whenever $|t|\leq T$ and $\Im z \geq b_0$.}
\label{A0 12}
\end{equation}
In this regime, $A(t, z;q)$ is given by the series expansion \eqref{936} and correspondingly,
\begin{align*}
\norm{ A( t,z;q ) }_{L^2_+\to L^\infty_+} &\lesssim \norm{ A_0(t, z) }_{L^2_+\to L^\infty_+}
	+ |t| \;\!\| A_0(t,z) q \|_{L^2_+\to L^\infty_+} \| C_+ \ol q A_0(t,z) \|_\op \\
&\lesssim \Bigl[ 1  + |t| \;\! \|q\|_{L^2}^2 \| A_0(t,z) \|_{L^1_+\to L^\infty_+} \Bigr]\norm{ A_0( t,z) }_{L^2_+\to L^\infty_+}  .
\end{align*}
Similarly, we obtain
\begin{align*}
\norm{ A( t,z;q ) }_{L^1_+\to L^2_+} &\lesssim 
\Bigl[ 1 + |t| \;\! \|q\|_{L^2}^2  \| A_0(t,z) \|_{L^1_+\to L^\infty_+} \Bigr] \norm{ A_0( t,z) }_{L^1_+\to L^2_+} 
\end{align*}
and
\begin{align*}
\norm{ A( t,z;q ) }_{L^1_+\to L^\infty_+} &\lesssim 
\Bigl[ 1 + |t| \;\! \|q\|_{L^2}^2  \| A_0(t,z) \|_{L^1_+\to L^\infty_+} \Bigr] \norm{ A_0( t,z) }_{L^1_+\to L^\infty_+} .
\end{align*}
The claims \eqref{A 4} and \eqref{A 444} then follow from \eqref{A0 5} and \eqref{A0 4}, albeit only for $\Im z\geq b_0$.

As the resolvent of a maximally accretive operator,
\begin{align}\label{1037}
A( z ;q) = A(z';q)  + (z-z') A(z';q) A(z;q)  \qtq{and} \| A(z;q) \|_\op \lesssim \tfrac{1}{\Im z} .
\end{align}
This extends \eqref{A 4} and \eqref{A 444} to general $\Im z > 0$ by choosing $z'=\Re z + ib_0$.  

Turning now to \eqref{1010}, we use the resolvent identity
\begin{align}\label{4.44}
A( z ;q) = A_0(z) \pm 2 t A_0(z) q C_+ \ol q A(z;q)  
\end{align}
together with \eqref{Schrod Poisson} and \eqref{A 4} to bound
\begin{align*}
\bigl| I_+\bigl[ A(t,z;q) f\bigr] \bigr| &\lesssim \tfrac{1}{\sqrt{\Im z}} \| f\|_{L^2} + |t|^{\frac12} \| qC_+ \ol qA(t,z;q) f\|_{L^1} \\
&\lesssim \tfrac{1}{\sqrt{\Im z}} \| f\|_{L^2} + \| q\|_{L^2}^2 \Bigl[\tfrac{1}{\sqrt{\Im z}} + \tfrac{1}{\Im z}\Bigr] \| f\|_{L^2} .
\end{align*}
This proves \eqref{1010}.
\end{proof}

We end this section by proving the explicit formula for $H^\infty_+(\R)$ solutions $q$ whose initial data satisfy $\langle x\rangle q^0\in L^2(\R)$.

\begin{proof}[Proof of Theorem~\ref{T:smooth EF}]
Let $q(t)$ be an $H^\infty_+$ solution to~\eqref{CM-DNLS} with initial data $q^0$ satisfying $\langle x\rangle q^0 \in L^2(\R)$.  A simple Gronwall argument guarantees that this decay property is preserved by the \eqref{CM-DNLS} evolution; see the discussion surrounding \eqref{1066 xq}.

Let $U(t;0)$ denote the unitary flow maps discussed in Proposition~\ref{P:Pflow}.  Given $\psi_0\in H^\infty_+(\R)$ satisfying $\langle x\rangle \psi_0 \in L^2(\R)$, consider
\begin{align}
\psi(t) := \bigl[ X U(t;0) - U(t;0)\bigl(X+2t\mc L_{q^0}\bigr) \bigr] \psi_0 \qtq{noting} \psi(0)=0.
\end{align}
The smoothness and decay hypotheses on $\psi_0$ and the fact that $U(t;0)$ preserves these properties ensure that $\psi(t)$ is well defined.  In particular, one may verify that $\psi(t)\in C_t L^2_+\cap C_t^1 H^{-2}_+$. In fact, using \eqref{comm 2} and \eqref{L conj}, the time derivative of $\psi(t)$ simplifies dramatically:
\begin{align*}
\tfrac{d}{dt} \psi(t) &= \mc P(q(t)) \psi(t) + \bigl[ X, \mc P(q(t))\bigr] U(t;0)\psi_0 - 2 U(t;0)\mc L_{q^0}\psi_0 =  \mc P(q(t)) \psi(t).
\end{align*}
By the uniqueness of solutions proved in Proposition~\ref{P:Pflow}, we deduce $\psi(t)\equiv 0$.  By the density of allowable $\psi_0$, we then deduce that $X U(t;0) = U(t;0)(X+2t\mc L_{q^0})$ and from there, that
\begin{align}\label{1222}
U(t;0) \bigl(X+2t\mc L_{q^0}-z\bigr)^{-1} = \bigl(X- z \bigr)^{-1} U(t;0) .
\end{align}

From \eqref{CIF} and the first relation in \eqref{L conj}, we know that
\begin{align}
 q(t,z) = \tfrac{1}{2\pi i} I_+\bigl[ \bigl(X- z \bigr)^{-1} q(t) \bigr]
 	= \tfrac{1}{2\pi i} I_+\bigl[ \bigl(X- z \bigr)^{-1} U(t;0) q^0 \bigr] .
\end{align}
Using \eqref{1222} and the first relation in \eqref{I+ def'} together with Lemma~\ref{t:P1=0}, we obtain
\begin{align*}
 q(t,z)  &= \tfrac{1}{2\pi i} I_+\bigl[ U(t;0) \bigl(X+2t\mc L_{q^0}-z\bigr)^{-1} q^0 \bigr] \\
 &= \tfrac{1}{2\pi i} \lim_{y\to\infty} \big\langle U(t;0)^*\chi_y, \bigl(X+2t\mc L_{q^0}-z\bigr)^{-1} q^0 \big\rangle \\
 &= \tfrac{1}{2\pi i}I_+\bigl[  \bigl(X+2t\mc L_{q^0}-z\bigr)^{-1} q^0\bigr],
\end{align*}
which proves \eqref{smooth EF}.
\end{proof}

\section{Well-posedness in $L^2_+(\R)$}\label{S:5}

In this section we will demonstrate global well-posedness of \eqref{CM-DNLS} in $L^2_+(\R)$.  For each initial data $q^0\in L^2_+(\R)$, we will construct the solution as a limit of solutions with smooth and well decaying initial data $q_n^0$.  Our concrete hypotheses will be these:
\begin{equation}\label{5hyp}
\begin{gathered}
q^0\in L^2_+(\R) \qtq{with} \| q^0\|_{L^2}^2 < M^*, \quad q^0_n\in H^\infty_+(\R), \quad \langle x\rangle q_n^0 \in L^2(\R),\\
q_n^0\to q^0 \text{ in $L^2$-sense}, \qtq{and} \sup_n \| q_n\|_{L^2} < M^* .
\end{gathered}
\end{equation}
Recall that $M^*$ denotes the equicontinuity threshold introduced in Definition~\ref{t:M* def}.  In Section~\ref{S:3}, we showed that $M^*=\infty$ in the defocusing case and that $M^*\geq 2\pi$ in the focusing case.


\begin{proposition}\label{t:q formula}
Suppose $q^0$ and $\{q^0_n\}_{n\in\N}$ satisfy \eqref{5hyp} and let $q_n(t)$ denote the solutions to \eqref{CM-DNLS} with initial data $q_n(0)=q_n^0$.  Then as $n\to\infty$, 
\begin{equation}
q_n(t,z)  \longrightarrow \efg (t,z; q^0) := \tfrac{1}{2\pi i} I_+\bigl[ (X + 2t\mc L_{q^0} - z)^{-1} q^0 \bigr]
\label{q formula}
\end{equation}
for each $t \in \R$ and each $\Im z > 0$.
\end{proposition}

\begin{proof}
In view of Theorem~\ref{T:smooth EF}, each $q_n(t,z)$ can be expressed through the explicit formula.  Thus, our task is so show that
\begin{equation}\label{5.3}
\lim_{n\to\infty} \bigl| I_+\bigl[ A(t,z; q_n^0)q_n^0 - A(t,z; q^0)q^0 \bigr] \bigr| = 0
\end{equation}
for fixed $t\in \R$ and $\Im z > 0$.  To this end, we set $Q = \{q^0\}\cup\{q_n^0 : n\in\N \}$ and adopt the abbreviations $A_n(t,z):=A(t,z; q_n^0)$ and $A(t,z)=A(t,z; q^0)$.

The set $Q$ is precompact in $L^2$ and so bounded and equicontinuous.  By Lemma~\ref{L:882}, we may choose $b_0$ sufficiently large so that for $\Im z\geq b_0$ we have
\begin{equation*}
\bnorm{ 2tC_+\ol{f}A_0(t,z)g }_{\op} \leq \tfrac{1}{2} \qtq{for all} f,g\in Q .
\end{equation*}

Combining this with \eqref{A0 5} and \eqref{A0 4}, we may then estimate
\begin{align*}
\bnorm{ \big( 2tq_n^0C_+\smash[b]{\ol q}_n^0 A_0(t,z) \big)^\ell & - \big( 2tq^0C_+\ol q^0A_0(t,z) \big)^\ell }_{L^2_+\to L^1_+} \\
&\qquad \lesssim 2\ell (\tfrac{1}{2})^{\ell-1} |t|^{\frac12} |\Im z|^{-\frac12} \norm{ q_n^0 - q^0 }_{L^2} \sup_{q\in Q} \norm{q}_{L^2},
\end{align*}
with the implicit constant independent of $\ell\geq 1$.   The role of this estimate is to allow us to expand both resolvents $A$ and $A_n$ in the manner of \eqref{936}.  Proceeding in this way and using \eqref{Schrod Poisson}, we find that
\begin{align}\label{1254}
\big| I_+\big[ & A_n(t,z)f - A(t,z) f\big] \big|  \notag\\
&\lesssim \sum_{\ell\geq 1} \Bigl| I_+\Bigl[ A_0(t,z)\Bigl\{ \big( 2tq_n^0C_+\smash[b]{\ol q}_n^0A_0(t,z) \big)^\ell - \big( 2tq^0C_+\ol q^0A_0(t,z) \big)^\ell \Big\} f \Bigr]\Bigr| \\
&\lesssim_Q |\Im z|^{-\frac12} \norm{ q_n^0 - q^0 }_{L^2}  \norm{f}_{L^2} \notag
\end{align}
for any $\Im z\geq b_0$ and any $f\in L^2_+(\R)$.

Combining \eqref{1254} and \eqref{1010}, we see that
\begin{align}\label{1254'}
\big| I_+\big[ A_n(z)q_n^0 - A(z) q^0\big] \big| \leq \big| I_+\circ \big[ A_n(z)- A(z)\big] q_n^0 \big| 
	+ \big| I_+\circ A(z) [q^0_n - q^0\big] \big| 
\end{align}
converges to zero as $n\to\infty$.  This proves \eqref{5.3} provided $\Im z\geq b_0$.

It remains to prove convergence for general $\Im z >0$.  The restriction $\Im z\geq b_0$ was needed for our treatment of the first term in RHS\eqref{1254'}.  To overcome this, we employ \eqref{1037} to expand
\begin{align*}
I_+\circ [A_n(t,z) - A(t,z)] &= I_+\circ \bigl[ A_n(t,z')-A(t,z') \bigr] \\
&\quad+ (z-z') I_+\circ \bigl[ A_n(t,z') - A(t,z') \bigr] A_n(t,z) \\
& \quad + (z-z') I_+\circ A(t,z') \bigl[ A_n(t,z) - A(t,z) \bigr] 
\end{align*}
in which we choose $z'=\Re z + i b_0$.  The contribution of the first two terms is easily handled using \eqref{1254} and \eqref{920}.
For the third term, we use \eqref{1010} in concert with Proposition~\ref{t:A'}, which showed that $A_n(z)\to A(z)$ in operator norm.
\end{proof}

By the conservation of mass and the Banach--Alaoglu Theorem, the solutions $q_n(t)$ converge subsequentially weakly in $L^2_+$ for each $t\in \R$. The previous proposition allows us to uniquely identify this limit.

\begin{corollary}\label{C:weak}
Suppose $q^0$ and $\{q^0_n\}_{n\in\N}$ satisfy \eqref{5hyp} and let $q_n(t)$ denote the solutions to \eqref{CM-DNLS} with initial data $q_n(0)=q_n^0$.  Then
\begin{equation}\label{1272}
q_n(t) \rightharpoonup \efg(t;q^0) \quad \text{weakly in $L^2_+(\R)$}.
\end{equation}
In particular, $\efg(t;q^0)$ defined by~\eqref{q formula} belongs to $L^2_+(\R)$.
\end{corollary}

\begin{proof}
Fix $t\neq 0$, and consider an arbitrary subsequence of $\{ q_n(t) \}_{n\in\N}$.  By conservation of mass, we know that
\begin{equation*}
\norm{q_n(t)}_{L^2} = \snorm{q_n^0}_{L^2} \lesssim \snorm{q^0}_{L^2} .
\end{equation*}
Therefore, by Banach--Alaoglu we may pass to a further subsequence along which
\begin{equation}\label{weak conv}
q_n(t) \rightharpoonup \wt{q}(t) \quad\text{weakly in }L^2_+ (\R)
\end{equation}
for some function $\wt{q}(t)$ in the Hardy space $L^2_+(\R)$. In particular $\wt q(t,z)$ is holomorphic in the upper half-plane.

As noted in \eqref{266}, evaluation at any point $z$ with $\Im z>0$ is a bounded linear functional on $L^2_+$. Correspondingly, $q_n(t,z) \to \wt{q}(t,z)$ pointwise.  In this way, Proposition~\ref{t:q formula} implies  $\wt q (t,z)   = \efg (t,z; q^0)$ for all $\Im z>0$.  This shows that $\efg (t,z; q^0)$ is holomorphic and lies in $L^2_+(\R)$.

Moreover, $\wt q (t,z)   = \efg (t,z; q^0)$ shows that the subsequential limit $\widetilde q\in L^2_+(\R)$ does not depend on the subsequence chosen.
Thus, we conclude that $q_n(t) \rightharpoonup \efg (t; q^0)$ along the whole sequence.
\end{proof}

With little additional effort, Corollary~\ref{C:weak} can be used to guarantee the existence of weak/distributional solutions to \eqref{CM-DNLS}.  To obtain well-posedness, however, we need to prove convergence in a stronger topology, one that guarantees that the limit depends continuously on both time and the initial data.  To this end, we seek to upgrade the weak convergence of $q_n(t)$ to $\efg (t,z;q^0)$ to strong convergence in $C_tL^2_x([-T,T]\times\R)$ for all $T>0$.  

\begin{theorem}\label{t:conv}
Suppose $q^0$ and $\{q^0_n\}_{n\in\N}$ satisfy \eqref{5hyp} and let $q_n(t)$ denote the solutions to \eqref{CM-DNLS} with initial data $q_n(0)=q_n^0$. Then for all $T>0$, the solutions $q_n(t)$ converge in $C_tL^2_x([-T,T]\times\R)$ to the function $\efg(t;q^0)$.
\end{theorem}

By Corollary~\ref{C:weak} and the Arzel\`a--Ascoli Theorem, to prove Theorem~\ref{t:conv} it suffices to show that the subset $\{ q_n(t): n\geq 1 \text{ and } |t|\leq T\}$ of $C_tL^2_x([-T,T]\times\R)$ is equicontinuous in both the time and space variables and tight in the space variable.  Recall that equicontinuity in the spatial variable was demonstrated in Section~\ref{S:3}.

Our next result proves equicontinuity in the time variable.
 
\begin{lemma}[Equicontinuity in time]\label{L:equi in t}
Suppose $q^0$ and $\{q^0_n\}_{n\in\N}$ satisfy \eqref{5hyp} and let $q_n(t)$ denote the solutions to \eqref{CM-DNLS} with initial data $q_n(0)=q_n^0$.   Then for each $T>0$ we have
\begin{equation}\label{equicty in t}
\sup_{n\geq 1} \norm{ q_n(t+h) - q_n(t) }_{C_tL^2_x([-T,T]\times\R)} \to 0 \quad\text{as }h\to 0.
\end{equation}
\end{lemma}

\begin{proof}
Throughout the proof, all spacetime norms will be over the slab $[-T,T]\times\R$.

For $N\geq1$ to be chosen later and $t,t+h \in [-T,T]$, we may estimate
\begin{align}
\norm{ q_n(t+h) - q_n(t) }_{C_tL^2_x}
&\leq \norm{ P_{\leq N}[q_n(t+h) -q_n(t)] }_{C_tL^2_x} + 2 \norm{P_{>N}q_n}_{C_tL^2_x}\notag\\
&\leq |h| \bigl\|\tfrac{d}{dt} P_{\leq N} q_n \bigr\|_{L^\infty_t L^2_x}+ 2\norm{P_{>N}q_n}_{C_tL^2_x}.\label{1:1}
\end{align}
As the set $\{ q_n(t): n\geq 1 \text{ and } |t|\leq T\}$ is equicontinuous in the spatial variable (see Section~\ref{S:3}), for any $\eps>0$ we may choose $N=N(\eps)\geq 1$ large enough so that 
\begin{align}\label{1:2}
\sup_{n\geq 1} \norm{P_{>N}q_n}_{C_tL^2_x}<\tfrac\eps{10}.
\end{align}

Using the equation \eqref{CM-DNLS}, the Bernstein inequality, and the conservation of mass, we may bound
\begin{align*}
\bigl\|\tfrac{d}{dt} P_{\leq N} q_n \bigr\|_{L^\infty_t L^2_x}
\lesssim N^2 \|q_n^0\|_{L^2} + \bigl\|P_{\leq N} [q_nC_+\big( |q_n|^2 \big)' ]\bigr\|_{L^\infty_t L^2_x}.
\end{align*}
Applying \eqref{277} and recalling the embedding $L^1\hookrightarrow H^{-1}$, we may bound the contribution of the nonlinearity as follows:
\begin{align*}
\bigl\|P_{\leq N} [q_nC_+\big( |q_n|^2 \big)'] \bigr\|_{L^\infty_t L^2_x}
&\lesssim N^5  \|q_nC_+\big( |q_n|^2 \big)' \|_{L^\infty_t H^{-5}_x}\\
&\lesssim N^5\|q_n\|_{L^\infty_tH^{-2}_x} \||q_n|^2\|_{L^\infty_t H^{-1}_x }\\
&\lesssim N^5 \|q_n^0\|_{L^2}^3 .
\end{align*}
Thus, 
\begin{equation*}
\bigl\|\tfrac{d}{dt} P_{\leq N} q_n \bigr\|_{L^\infty_t L^2_x} \lesssim N^2 \|q_n^0\|_{L^2} + N^5 \|q_n^0\|_{L^2}^3,
\end{equation*}
uniformly for $n\geq 1$.  Choosing $|h|$ sufficiently small, we may ensure that 
$$
\sup_{n\geq 1} |h| \bigl\|\tfrac{d}{dt} P_{\leq N} q_n \bigr\|_{L^\infty_t L^2_x}<\tfrac\eps{10}.
$$
Combining this with \eqref{1:1} and \eqref{1:2} and recalling that $\eps>0$ was arbitrary, completes the proof of the lemma.
\end{proof}

It remains to prove tightness in the spatial variable.  This will be accomplished in two distinct steps.  First, we will show that for any $b>0$, the functions $x\mapsto q_n(t,x+ib)$ are tight in $L^2(\R)$; this is proved in Proposition~\ref{t:tight 1} below.  In the second step, we will upgrade this statement to tightness on the real line ($b=0$); this is realized in Proposition~\ref{t:tight 2}.

\begin{proposition}\label{t:tight 1}
Suppose $q^0$ and $\{q^0_n\}_{n\in\N}$ satisfy \eqref{5hyp} and let $q_n(t)$ denote the solutions to \eqref{CM-DNLS} with initial data $q_n(0)=q_n^0$.  For any $b>0$ and $T>0$, the set
\begin{equation*}
\big\{ q_n(t,\cdot+ib) : n\geq 1,\ |t|\leq T \big\}
\end{equation*}
is tight in $L^2(\R)$.
\end{proposition}

\begin{proof}
By Theorem~\ref{T:smooth EF}, the functions $q_n(t,z)$ admit the representation \eqref{smooth EF} for any $z$ with $\Im z>0$.  Together with the resolvent identity, this leads to
\begin{align*}
q_n(t,z) = \tfrac{1}{2\pi i} I_+\big[A_0(t,z)q_n^0 \big]\pm \tfrac{1}{2\pi i} I_+\big[ A_0(t,z) 2tq_n^0C_+\ol q_n^0A(t,z;q_n^0) q_n^0 \big].
\end{align*}

We will analyze the two summands above using the identity \eqref{A0 6}.  To begin, we write
\begin{equation}
\tfrac{1}{2\pi i} I_+\big[A_0(t,z)q_n^0 \big] = \big[ e^{it\Delta}q_n^0 \big](z).
\label{tight 1}
\end{equation}

Let $Q:=\{q^0_n: n\geq 1\}$.  As $Q$ is precompact in $L^2(\R)$ and the free Schr\"odinger propagator is continuous on $L^2(\R)$, the set of functions
\begin{equation*}
\ms{F}:=\big\{ \big[ e^{it\Delta}q\big](x) : \ q\in Q,\ |t|\leq T \big\}
\end{equation*}
is also precompact in $L^2(\R)$.  

For $f\in\ms{F}$ and $b>0$, we may use the Poisson integral formula
\begin{equation}
f(x+ib) = [e^{-b|\partial|}f](x) = \tfrac{1}{\pi} \int \tfrac{b}{(y-x)^2+b^2} f(y)\,dy
\label{PIF}
\end{equation}
and Cauchy--Schwarz to estimate
\begin{align*}
\int_{|x|\geq L} |f(x+ib)|^2\,dx
&\lesssim \int_{|x|\geq L} \int_{\R} \tfrac{b}{(y-x)^2+b^2} |f(y)|^2\,dy\,dx \\
&\lesssim \int_{|y|\geq L_0} |f(y)|^2\,dy + \tfrac{b}{L-L_0} \norm{f}_{L^2}^2 ,
\end{align*}
for any $L\geq 2L_0\geq 1$.  As $\ms{F}$ is precompact in $L^2(\R)$, we may pick $L_0$ large and then $L\gg L_0$ to make the right-hand side above arbitrarily small.  In view of \eqref{tight 1}, this demonstrates that the set
$$\big\{ I_+\big[A_0(t,x+ib)q \big]:\ q\in Q,\ |t|\leq T\big\}$$
is tight in $L^2(\R)$.

It remains to show that the functions
\begin{align*}
\tfrac{1}{2\pi i} I_+\big[ A_0(t,x+ib) 2tq C_+\ol q &A(t,x+ib;q) q \big]\! \!=\!\! \Bigl[ e^{it\Delta}e^{-b|\partial|}\bigl( 2tqC_+\ol{q}A(t,x+ib;q)q  \bigr)\Bigr](x)
\end{align*}
with $q\in Q$ and $|t|\leq T$ form a set that is tight in $L^2(\R)$.

We first observe that the set
\begin{equation*}
\ms{G} = \big\{ 2tqC_+\ol{q}A(t,x+ib;q)q : \ q\in Q,\ |t|\leq T \big\}
\end{equation*}
is bounded and tight in $L^1(\R)$.  Indeed, using \eqref{A 4} we may bound
\begin{align*}
\bigl\|  2tqC_+\ol{q}A(t,x+ib;q_n)&q\bigr\|_{L^1(|x|\geq L)}\\
&\lesssim |t| \|q\|_{L^2(|x|\geq L)} \|q\|_{L^2}\|A(t,x+ib;q)\|_{L^2_+\to L^\infty_+} \|q\|_{L^2}\\
&\lesssim \sqrt{T} \bigl( \tfrac1{\sqrt b} + \tfrac1b\bigr)  \|q\|_{L^2(|x|\geq L)}\|q\|_{L^2}^2
\end{align*}
uniformly for $L\geq 0$, $q\in Q$, $|t|\leq T$, and $b>0$.

In view of the estimate
\begin{equation*}
\bnorm{ e^{-b|\partial|}g }_{H^1}^2
= \int (1+\xi^2)e^{-2b|\xi|} |\wh{g}(\xi)|^2\,d\xi
\lesssim_b \snorm{ \wh{g} }_{L^\infty}^2
\lesssim_b \norm{g}_{L^1}^2,
\end{equation*}
we deduce that the set
\begin{equation*}
\ms{H} = \big\{ e^{-b|\partial|}g : g\in\ms{G} \big\}
\end{equation*}
is bounded in $H^1(\R)$, and so it is bounded and equicontinuous in $L^2(\R)$.  The set $\ms{H}$ is also tight in $L^2(\R)$.  To see this, we may use the Poisson integral formula~\eqref{PIF} and Minkowski's inequality to estimate
\begin{align*}
\bnorm{e^{-b|\partial|}g }_{L^2(|x|\geq L)}
&\lesssim \int_{|y|\geq L_0} \bnorm{ \tfrac{b}{x^2+b^2} }_{L^2} \,|g(y)|\, dy \\
&\quad + \int_{|y|\leq L_0} \bnorm{ \tfrac{b}{x^2+b^2} }_{L^2(|x|\geq L-L_0)}\, |g(y)|\, dy \\
&\lesssim \int_{|y|\geq L_0} \tfrac1{b^{1/2}}|g(y)|\,dy + \tfrac{b}{(L-L_0)^{3/2}} \norm{g}_{L^1},
\end{align*}
for any $L\geq 2L_0\geq 1$.  As $\ms{G}$ is bounded and tight in $L^1(\R)$, we may pick $L_0$ large and then $L\gg L_0$ to make the right-hand side above arbitrarily small.

Thus $\ms{H}$ is precompact in $L^2(\R)$.  As the free Schr\"odinger propagator is continuous on $L^2(\R)$, we conclude that the set
$$
\big\{I_+\big[ A_0(t,x+ib) 2tq C_+\ol q A(t,x+ib;q) q\bigr]:\ n\geq 1, \ |t|\leq T\big\}
$$
is precompact in $L^2(\R)$ and so tight in $L^2(\R)$.  This completes the proof of the proposition.
\end{proof}

\begin{proposition}[Tightness in space]\label{t:tight 2}
Suppose $q^0$ and $\{q^0_n\}_{n\in\N}$ satisfy \eqref{5hyp} and let $q_n(t)$ denote the solutions to \eqref{CM-DNLS} with initial data $q_n(0)=q_n^0$.  Given $T>0$, the set $\{ q_n(t): n\geq 1 \text{ and } |t|\leq T\}$ is tight in $L^2(\R)$.
\end{proposition}

\begin{proof}
For a frequency cutoff $N\geq 1$ to be chosen later, we decompose $q_n = P_{\leq N}q_n + P_{>N}q_n$.  Let $\chi : \R\to [0,1]$ be a smooth cutoff function equal to $1$ on $|x|\geq 2$ and to $0$ on $|x|\leq 1$.  For $R>0$, we define $\chi_{>R}(x) := \chi(\frac{x}{R})$.  We may then bound
\begin{equation*}
\norm{ \chi_{>R} q_n(t)}_{L^2}
\leq \norm{ \chi_{>R} P_{\leq N}q_n(t) }_{L^2} + \norm{ P_{>N} q_n(t)}_{L^2}.
\end{equation*}
The hypotheses \eqref{5hyp} guarantee $L^2$-equicontinuity of the solutions $q_n(t)$.  Therefore, we can choose $N\geq 1$ sufficiently large to render
 $\sup_{n\geq 1}\sup_{|t|\leq T} \norm{ P_{>N} q_n(t)}_{L^2}$ arbitrarily small.  

To demonstrate tightness of the low frequencies, for $b>0$ fixed we write
\begin{equation*}
P_{\leq N} q_n(t)= \big( P_{\leq N} e^{b|\partial|} \big) \big( e^{-b|\partial|}q_n(t) \big)  = K * \big( e^{-b|\partial|}q_n(t) \big) ,
\end{equation*}
where the convolution kernel $K$ is the Schwartz function
\begin{equation*}
K(x)
= \tfrac{1}{\sqrt{2\pi}} \int e^{ix\xi} e^{b|\xi|} \varphi\big( \tfrac{\xi}{N} \big)\, d\xi
= \sqrt{ \tfrac{2}{\pi} } \int_0^\infty \cos(x\xi) e^{b\xi} \varphi\big( \tfrac{\xi}{N} \big)\, d\xi
\end{equation*}
and $\varphi : \R \to [0,1]$ denotes the cutoff function used to define the Littlewood--Paley projections.  Integrating by parts twice yields the easy bound
\begin{equation*}
K(x) \lesssim_{b,N} \langle x \rangle^{-2} .
\end{equation*}
Using this and the observation that $\chi_{>R} [(\chi_{\leq \frac  R4}K)* (\chi_{\leq \frac R4}q_n)]=0$, which follows from support considerations, we may estimate
\begin{align*}
\norm{ \chi_{>R} P_{\leq N} q_n(t)}_{L^2} 
&\leq \bnorm{ K * \big[ \chi_{>\frac R4}e^{-b|\partial|}q_n(t) \big] }_{L^2} + \bnorm{ \big[ \chi_{>\frac R4}K\big] * \big[ \chi_{\leq \frac R4}e^{-b|\partial|}q_n(t) \big] }_{L^2} \\
&\lesssim \snorm{ \chi_{>\frac R4}e^{-b|\partial|}q_n(t) }_{L^2} + R^{-1} \snorm{\chi_{\leq \frac R4}e^{-b|\partial|}q_n(t) }_{L^2} \\
&\lesssim \snorm{ \chi_{>\frac R4}e^{-b|\partial|}q_n(t) }_{L^2} + R^{-1} \|q_n^0\|_{L^2}.
\end{align*}
By Proposition~\ref{t:tight 1}, the functions $[e^{-b|\partial|}q_n](t,x) = q_n(t, x + ib)$ form an $L^2$-tight set for $n\geq 1$ and $|t|\leq T$.  Therefore, we may choose $R$ sufficiently large to make the right-hand side above arbitrarily small, which demonstrates that $\{P_{\leq N}q_n(t):\ n\geq 1,\ |t|\leq T\}$ is tight in $L^2(\R)$, as desired.
\end{proof}

We are now ready to prove well-posedness of \eqref{CM-DNLS} in $L^2_+(\R)$ for data satisfying $M(q^0)<M^*$, with $M^*$ denoting the equicontinuity threshold introduced in Definition~\ref{t:M* def}.  This will complete the proof of both Theorem~\ref{t:wpd} and Theorem~\ref{t:wpf}. 

\begin{proof}[Proof of well-posedness in $L^2_+(\R)$]
We will show that the data-to-solution map for \eqref{CM-DNLS} extends uniquely from $\{q\in H^\infty_+(\R):\, \langle x\rangle q \in L^2(\R) \text{ and }M(q)<M^*\}$ to a jointly continuous map $\Phi : \R\times \{q\in L^2_+(\R):\ M(q)<M^*\} \to L^2_+$.  

Given initial data $q^0\in L^2_+(\R)$ with $M(q^0)<M^*$, let $\{ q_n^0 \}_{n\geq 1}$ be a sequence satisfying \eqref{5hyp}.  Applying Theorem~\ref{t:conv} to the sequence $\{ q_n^0 \}_{n\geq 1}$, we see that the corresponding $H^\infty_+(\R)$ solutions $q_n(t)$ to~\eqref{CM-DNLS} converge in $L^2(\R)$ and the limit is independent of the sequence $\{ q_n^0 \}_{n\geq 1}$.  Consequently,
\begin{equation*}
\Phi(t,q^0) := \lim_{n\to\infty} q_n(t)
\end{equation*}
is well-defined.
	
We must show that $\Phi$ is jointly continuous.  Fix $T>0$ and $q^0\in L^2_+(\R)$ with $M(q^0)<M^*$.  Let $\{ q_n^0 \}_{n\geq 1}$ be an $L^2_+$ sequence that converges to $q^0$ in $L^2(\R)$; without loss of generality, we may assume that $\sup_{n\geq 1} M(q_n^0)<M^*$.  By the definition of $\Phi$, we may choose another sequence $\wt{q}_n(t)$ of $H^\infty_+$ solutions to~\eqref{CM-DNLS} so that $\langle x\rangle \wt{q}_n(0)\in L^2(\R)$ and
\begin{equation}\label{wp 1}
\sup_{|t|\leq T}  \norm{ \Phi(t,q_n^0) - \wt{q}_n(t) }_{L^2} \to 0 \quad\text{as }n\to\infty .
\end{equation}
In particular, $\wt{q}_n(0) \to q^0$ in $L^2(\R)$, and so Theorem~\ref{t:conv} yields
\begin{equation}\label{wp 2}
\sup_{|t|\leq T}  \norm{ \wt{q}_n(t) - \Phi(t,q^0) }_{L^2} \to 0 \quad\text{as }n\to\infty .
\end{equation}

Given $\{t_n\}_{n\geq 1}\subset[-T,T]$ that converges to some $t\in [-T,T]$, we may bound
\begin{align*}
&\norm{ \Phi(t_n,q_n^0) - \Phi(t,q^0) }_{L^2}\\
&\qquad\leq \norm{ \Phi(t_n,q_n^0) - \wt{q}_n(t_n) }_{L^2} +  \norm{ \wt{q}_n(t_n)-\wt{q}_n(t)  }_{L^2} + \norm{\wt{q}_n(t)  - \Phi(t,q^0) }_{L^2}\\
&\qquad\leq \sup_{|t|\leq T}  \norm{ \Phi(t,q_n^0) - \wt{q}_n(t) }_{L^2} + \norm{ \wt{q}_n(t_n)-\wt{q}_n(t) }_{L^2} +\sup_{|t|\leq T} \norm{ \wt{q}_n(t) - \Phi(t,q^0) }_{L^2}.
\end{align*}
The right-hand side above converges to zero as $n\to\infty$ by \eqref{wp 1}, \eqref{wp 2}, and \eqref{equicty in t}.  This demonstrates that $\Phi$ is jointly continuous.
\end{proof}

Finally, we prove the explicit formula for general $L^2_+(\R)$ initial data:

\begin{proof}[Proof of Theorem~\ref{t:magic formula}]
From Theorem~\ref{T:smooth EF}, we know that the explicit formula holds for initial data that is smooth and well-decaying.  With this in mind, we choose a sequence $q^0_n$ of such initial data that converges to $q^0$ in $L^2$-sense and satisfies $M(q^0_n)<M^*$.   Now we simply need to show convergence of both sides of \eqref{smooth EF}.  In the case of the left-hand side, this is immediate from the well-posedness just proved and \eqref{266}.

The right-hand side requires further discussion.  In view of \eqref{Schrod Poisson} and \eqref{4.44}, we merely need to show that
\begin{align*}
2 t q^0_n C_+ \smash[b]{\ol q}^0_n A(t,z;q^0_n) q^0_n \to 
	2 t q^0 C_+ {\ol q}^0 A(t,z;q^0) q^0 \quad \text{in $L^1$ sense.}
\end{align*}
Employing \eqref{A 4}, elementary inequalities reduce this question to showing that
\begin{align*}
\lim_{n\to\infty}\ \bigl\| A(t,z;q^0_n) - A(t,z;q^0)\bigr\|_{L^2_+\to L^\infty_+} =0.
\end{align*}
This in turn follows from \eqref{A 4}, \eqref{A 444}, and the resolvent identity.
\end{proof}

\section{Well-posedness in $H^s_+(\R)$}\label{S:6}

The goal of this section is to prove 

\begin{theorem}[Global well-posedness in $H^s_+(\R)$]\label{T:gwp in Hs}
Fix $0<s<1$. Both \eqref{CM-DNLS} equations {\slshape(}focusing and defocusing{\slshape)} are globally well-posed in the space
\begin{equation*}
B_{M^*} = \{ q\in H^s_+(\R) : \ M(q)< M^* \},
\end{equation*}
where $M^*$ denotes the equicontinuity threshold introduced in Definition~\ref{t:M* def}.
\end{theorem}

In view of the $L^2_+(\R)$ global well-posedness result demonstrated in the preceding section, it will suffice to prove that \eqref{CM-DNLS} solutions satisfy a priori $H^s$ bounds, and that the orbits of $H^s$-bounded and equicontinuous sets of initial data form an $H^s$-equicontinuous set.  These two claims will be taken up in Proposition~\ref{P:s equi}.  The proof of this proposition is based on a novel argument, in which the following family of functions play a central role:

\begin{definition}
Fix $0<s<1$.  For each $K\geq 0$, we define the function 
\begin{align}\label{F defn}
F_K(E) := \frac{\sin(\pi s)}{\pi}\int_K^\infty \frac{E \lambda^{s-1}\,d\lambda}{E+\lambda} \qtq{for each} E\in[0,\infty).
\end{align}
\end{definition}

The particular structure of these functions is dictated by our use of Loewner's Theorem on operator monotone functions.  Concretely, for any positive definite operators $A\geq B\geq 0$, it follows that $F_K(A)\geq F_K( B) \geq 0$.

If one chooses $K=0$, then \eqref{F defn} can be evaluated exactly: $F_0(E) = E^s$.  It is to ensure this identity that we have included the prefactor in \eqref{F defn}.  Clearly $F_K(E)$ is a decreasing function of $K$; thus $F_K(E)\leq E^s$ for all $E\geq 0$ and all $K\geq 0$.

For $K>0$, we know of no explicit formula for $F_K(E)$; however, it is not difficult to describe the order of magnitude of this function:
\begin{align}\label{F size}
F_K(E) \approx \begin{cases} E^s  &: E\geq K \\ E K^{s-1} &: 0\leq E \leq K\end{cases}
\end{align}
uniformly in $K\geq 0$.  The overall shape described by \eqref{F size} also plays an important role in the selection of the function $F_K(E)$, through the following expression of equicontinuity:

\begin{lemma}\label{L:6.3}
A bounded subset $Q\subseteq H^s_+(\R)$ is $H^s$-equicontinuous if and only if
\begin{align}\label{F equi}
  \sup_{q\in Q} \ \bigl\langle q,  F_K\bigl( (\mc L_0 + \kappa)^2 \bigr) q \bigr\rangle  \to 0  \qtq{as} K\to\infty
\end{align}
for any single $\kappa>0$.
\end{lemma}

\begin{proof}
In view of \eqref{F size}, we have
$$
\bigl\langle q,  F_K\bigl( (\mc L_0 + \kappa)^2 \bigr) q \bigr\rangle
	\lesssim \bigl(\tfrac{\kappa^2+N^2}{K}\bigr)^{1-s} \bigl\| (\mc L_0 + \kappa)^s q_{\leq N} \bigr\|_{L^2}^2
		+ \bigl\|  (\mc L_0 + \kappa)^s q_{>N} \bigr\|_{L^2}^2 .
$$
Choosing $N$ large and then sending $K\to\infty$, we see that equicontinuous sets satisfy \eqref{F equi}.  Conversely,
$$
\| q_{>N} \|_{H^s}^2  \lesssim \bigl\langle q,  F_K\bigl( (\mc L_0 + \kappa)^2 \bigr) q \bigr\rangle  \qtq{when} K=(N+\kappa)^2 .
$$
Thus, \eqref{F equi} implies that $Q$ is equicontinuous.
\end{proof}

The use of $\mc L_0$ in this lemma, makes for an easy connection to $H^s$-equicontinuity; however, it is the Lax operator $\mc L_{q(t)}$ associated to the solution itself that enjoys a close connection with the \eqref{CM-DNLS} flows.  Our next result bridges this divide.

\begin{lemma}\label{L:6.4}
Fix $s\in(0,1)$.  Suppose $Q\subseteq H^\infty_+(\R)$ is bounded and equicontinuous in $L^2(\R)$.  Then there exists $\kappa_0=\kappa_0(Q)$ so that
\begin{align}\label{F2F}
\bigl\langle q,  F_K\bigl( (\mc L_q + \kappa)^2 \bigr) q \bigr\rangle 
	\approx \bigl\langle q,  F_K\bigl( (\mc L_0 + \kappa)^2 \bigr) q \bigr\rangle  	
\end{align}
uniformly for $q\in Q$,  $\kappa\geq\kappa_0$, and $K\geq 0$.
\end{lemma}

\begin{proof}
Using \eqref{rc equi}, we may choose $\kappa_0(Q)>0$ so that
\begin{align}\label{1614}
\sup_{q\in Q} \bigl\| q C_+ \ol q R_0(\kappa) \bigr\|_\op \leq \tfrac14 \quad\text{uniformly for $q\in Q$ and $\kappa\geq\kappa_0$}.
\end{align}
Using elementary manipulations and \eqref{1614}, we find that
\begin{align*}
\Bigl| \bigl\langle f, (\mc L_{q} + \kappa)^{2}f\bigr\rangle -  \bigl\langle f, (\mc L_0 + \kappa)^{2}f\bigr\rangle \Bigr|
	&\leq \bigl\| q C_+ \ol q f \bigr\|_{L^2} \bigl[ 2\|(\mc L_0 + \kappa)f \|_{L^2}  +\|q C_+ \ol q f  \|_{L^2} \bigr]\\
&\leq  \tfrac9{16} \|(\mc L_0 + \kappa)f \|_{L^2}^2
\end{align*}
for all $\kappa\geq \kappa_0$ and any $f\in H^1_+(\R)$.  Equivalently,  
\begin{align*}
\tfrac7{16} (\mc L_0 + \kappa)^{2} \leq (\mc L_{q} + \kappa)^{2} \leq \tfrac{25}{16} (\mc L_0 + \kappa)^{2}
	\quad\text{for all $q\in Q$ and $\kappa\geq\kappa_0$}.
\end{align*}
As the functions $F_K$ are operator monotone, we deduce that
\begin{align*}
F_K\bigl(  \tfrac7{16} (\mc L_0 + \kappa)^{2} \bigr)
	\leq F_K\bigl(  (\mc L_{q} + \kappa)^{2} \bigr)
	\leq F_K\bigl( \tfrac{25}{16} (\mc L_0 + \kappa)^{2} \bigr).
\end{align*}
From the overall shape of $F_K(E)$ described in \eqref{F size}, it then follows that
\begin{align}
c F_K\bigl(  (\mc L_0 + \kappa)^{2} \bigr)
	\leq F_K\bigl(  (\mc L_{q} + \kappa)^{2} \bigr)
	\leq c^{-1} F_K\bigl( (\mc L_0 + \kappa)^{2} \bigr).
\end{align}
for some (small) number $c$ depending only on $s$.  The claim \eqref{F2F} is a direct consequence of this.
\end{proof}

\begin{proposition}\label{P:s equi}
Fix $s\in(0,1)$.  Assume $Q\subset \{q\in H^\infty_+(\R):\ M(q)<M^*\}$ is bounded and equicontinuous in $H^s(\R)$.  Then
\begin{equation}\label{6:Q*}
Q^* := \{ e^{tJ\nabla\ham_\pm} q^0 : q^0\in Q,\ t\in\R \} 
\end{equation}
is also bounded and equicontinuous in $H^s(\R)$.
\end{proposition}

\begin{proof}
As the set $Q$ is bounded in $H^s(\R)$, it is bounded and equicontinuous in $L^2(\R)$.  By the results of Section~\ref{S:3}, the set of orbits $Q^*$ is bounded and equicontinuous in $L^2(\R)$.  Thus, Lemma~\ref{L:6.4} may be applied to $Q^*$;  we define $\kappa_0=\kappa_0(Q^*)$ accordingly.

Employing both parts of \eqref{L conj} from Proposition~\ref{P:Pflow}, we see that
\begin{align}\label{FK cons}
\big\langle q(t), F_K\bigl((\mc L_{q(t)} + \kappa)^{2}\bigr) q(t) \big\rangle
	=  \big\langle q(0), F_K\bigl((\mc L_{q(0)} + \kappa)^{2}\bigr) q(0) \big\rangle,
\end{align}
for any choices of $\kappa\geq \kappa_0(Q^*)$, $s\in(0,1)$, $K\geq 0$, and any $H^\infty_+(\R)$ solution  $q(t)$ to \eqref{CM-DNLS}.

Freezing $\kappa=\kappa_0(Q^*)$ and applying Lemma~\ref{L:6.4}, the identity \eqref{FK cons} yields
\begin{align}\label{FK cons'}
\big\langle q(t), F_K\bigl((\mc L_0 + \kappa)^{2}\bigr) q(t) \big\rangle
	\approx  \big\langle q(0), F_K\bigl((\mc L_0 + \kappa)^{2}\bigr) q(0) \big\rangle,
\end{align}
uniformly for $q(0)\in Q$, $t\in\R$, and $K\geq0$.

Recalling that $F_0(E)=E^s$, the $K=0$ case of \eqref{FK cons'} shows that $Q^*$ inherits $H^s$-boundedness from $Q$.  Moreover, sending $K\to\infty$ in \eqref{FK cons'} and employing Lemma~\ref{L:6.3}, we see that $Q^*$ inherits $H^s$-equicontinuity from $Q$.
\end{proof}

We are now ready to finish the proof of well-posedness in $H^s_+(\R)$.  Based on the tools we have developed, it is most convenient to construct the data-to-solution map by extension from the class of smooth solutions.  As the extension will be shown to be $H^s_+(\R)\to C([-T,T];H^s_+(\R))$ continuous, it coincides with the restriction to $H^s_+(\R)$ of the map constructed in the previous section for initial data in $L^2_+(\R)$. 

\begin{proof}[Proof of Theorem~\ref{T:gwp in Hs}]
Suppose $q^0\in H^s_+(\R)$ and $M(q^0)<M^*$.  Let $q_n(t)$ be a sequence of $H^\infty_+(\R)$ solutions with $q_n(0) \to q^0$ in $H^s$-sense.  Our sole obligation is to show that $q_n(t)$ converges in $C([-T,T];H^s_+(\R))$ for every finite $T>0$.  If such convergence did not hold, there would be a convergent sequence $t_n$ so that $q_n(t_n)$ did not converge in $H^s_+(\R)$.  We will refute this.

From the results of the previous section, we know that $q_n(t_n)$ converges in $L^2_+(\R)$.  Moreover, Proposition~\ref{P:s equi} shows that the sequence $q_n(t_n)$ is both uniformly bounded and equicontinuous in $H^s_+(\R)$.  Thus, $q_n(t_n)$ does indeed converge in $H^s_+(\R)$.
\end{proof}

\bibliography{refs}

\end{document}